\documentclass[a4paper,12pt]{article}
\usepackage{amsthm}

\usepackage[T1]{fontenc} 
\usepackage[utf8]{inputenc}
\usepackage[ english]{babel}
\usepackage{amsmath}
\usepackage{amsfonts}
\usepackage{enumerate}
\usepackage{amssymb}
\usepackage[usenames]{color}
\usepackage{graphicx}
\usepackage{amsmath,dsfont}
\usepackage{hyperref}
\hypersetup{colorlinks=true, linkcolor=black}% liens en noir : seules les corrections apparaissent en bleu

\newtheorem{theo}{Theorem}[section]
\newtheorem{lemm}[theo]{Lemma}
\newtheorem{coro}[theo]{Corollary}

\newtheorem{prop}[theo]{Proposition}
\newtheorem{rema}[theo]{Remark}
\newtheorem{defi}[theo]{Definition}

\newtheorem{properties}[theo]{Properties}

\newcommand{\R}{\mathbb{R}}
\numberwithin{equation}{section}
\newcommand{\N}{\mathbb{N}}
\newcommand{\Z}{\mathbb{Z}}

\def\C{\mathcal{C}}

\def\w{\textit {\textbf w}}

\def\u{\textit {\textbf u}}
\def\v{\textit {\textbf v}}

\def\<{\langle}
\def\>{\rangle}

\def\z{{\textbf {\textit z}}}

\def\0{{\bf 0}}

\def\pup1{{\partial\u\over\partial x_1}}

\usepackage{graphicx}
\usepackage{graphics}
\usepackage{epsfig}
\usepackage{amsmath,fourier}
\usepackage{amsfonts}
\usepackage{psfrag}
\usepackage{subfigure,verbatim}

\usepackage[a4paper]{geometry}
\title{\bf{\textsl{$L^p$-theory for the exterior Stokes problem with Navier's type slip-without-friction boundary conditions } }}

\author{Anis Dhifaoui \\}
\smallskip
\date{
{\small LR Analysis and Control of PDEs (LR22ES03)\\
Department of Mathematics, Faculty of Sciences of Monastir\\
University of Monastir, 5019 Monastir, Tunisia}\\
\bigskip
{\small anisdhifaoui123@gmail.com}}
\begin{document}

\maketitle

\begin{abstract}
In this paper, we consider the stationary Stokes equations in an exterior domain three-dimensional  under a slip boundary condition without friction. We set the problem in weighted Sobolev spaces in order to control the behavior at infinity of the solutions. In this work, we try to investigate the existence and uniqueness up to an element of a nontrivial kernel of the weak solutions related to this problem in $L^p-$theory when $p>3$. Our {proofs} are based on obtaining inf-sup conditions that play a fundamental role.\\


\end{abstract}
\section{Introduction}
In this paper, we are interested in the Stationary Stokes equations, this problem describes the flow of a viscous and incompressible fluid past an obstacle:
\begin{equation}
\label{Stokes}
\begin{split}
-\Delta\textit{\textbf{u}}+\nabla\pi&=\textit{\textbf{f}}\quad\text{in}\quad\Omega,\\
\mathrm{div }\,\textit{\textbf{u}}&=0\quad\text{in}\quad\Omega.
\end{split}
\end{equation} 
In~\eqref{Stokes}, the flow domain $\Omega$ is a three-dimensional external domain, i.e. the complement of a bounded domain which represents the obstacle,  
the unknowns are the velocity of the fluid $\textit{\textbf{u}}$ and the pressure $\pi$ and $\textit{\textbf{f }}$ the external forces acting on the fluid. For many years, the standard boundary conditions are the Dirichlet or no-slip boundary condition, has been the most widely used, given its success in reproducing the standard velocity profiles for incompressible viscous fluid. These latter conditions prescribe that the fluid adheres to the boundary of the domain. More generally, in the case where the obstacles have an approximate limit, the standard conditions are no longer valid (see for example~\cite{Serrin_book}). Therefore, another approach has been introduced, which assumes that due to the roughness of the boundary and the viscosity of the fluid, there is a stagnant fluid layer near the boundary, which allows the fluid to slip (see for example~\cite{casado}). To overcome the complicated description of the problem, the Navier slip boundary conditions are commonly used.

\noindent These conditions, proposed by Navier~\cite{NAVIER}, can be written as
\begin{equation}
\label{Navier.BC}
\textbf{\textit{u}}\cdot\textbf{\textit{n}}=0,\quad \big((\mathrm{\textbf{D}}\textbf{\textit{u}})\textbf{\textit{n}}+\alpha\textbf{\textit{u}}\big)_{\tau}=\boldsymbol{0}
\quad\text{on}\quad\Gamma,
\end{equation}
where 
\begin{equation*}
\label{def.deformation}
\textbf{D}\textbf{\textit{u}}=\dfrac{1}{2}\left(\nabla\,\textbf{\textit{u}}+\nabla\,\textbf{\textit{u}}^T\right)
\end{equation*}
denotes the rate-of-strain tensor field, $\alpha$ is a scalar friction function, $\textbf{\textit{n}}$ is the unit normal vector to $\Gamma$
 and the notation $(\cdot)_\tau$ denotes the tangential component of a vector on $\Gamma$. The first condition in~\eqref{Navier.BC} is the no-penetration condition and the second condition expresses the fact that the tangential 
velocity is proportional to the tangential stress.\\

\noindent Now, let us recall some notations related to the boundary condition. First, for any vector field $\textbf{\textit{u}}$ on $\Gamma$, we can write
\begin{equation*}
\label{decomposition.trace}
\textbf{\textit{u}}=\textbf{\textit{u}}_{\tau}+(\textbf{\textit{u}}\cdot\textbf{\textit{n}})\textbf{\textit{n}},
\end{equation*}
where $\textbf{\textit{u}}_{\tau}$ is the projection of $\textbf{\textit{u}}$ on the tangent {hyperplane} to $\Gamma$. Next, for any  point  $\textbf{\textit{x}}$ on $\Gamma$, one may choose an open neighbourhood $\mathcal{W}$ of $\textbf{\textit{x}}$ in $\Gamma$ small enough to allow the existence of two families of $\mathcal{C}^2$ curves on $\mathcal{W}$ and where the lengths $s_{1}$ and $s_{2}$ along each family of curves are possible system of coordinates. Denoting by $\boldsymbol{\tau}_1$, $\boldsymbol{\tau}_2$ the unit tangent vectors to each family of curves, we have
 $$\textbf{\textit{u}}_{\tau}=(\textbf{\textit{u}}\cdot\boldsymbol{\tau}_{1})\boldsymbol{\tau}_{1}+(\textbf{\textit{u}}\cdot\boldsymbol{\tau}_{2})\boldsymbol{\tau}_{2}.
$$
For a mathematical analysis of the Stokes system satisfying~\eqref{Navier.BC}, the first {pioneering} paper is due to Solonnikov and Scadilov~\cite{Solonnikov_TMIS_1973} with $\alpha=0$. More recently, Beirao da Veiga~\cite{Veiga_2005} proved existence results for weak and strong solutions in the $L^2$ setting. However, we can prove that 
\begin{eqnarray*}\label{tenseur de déformation sur le bord}
2\big((\mathrm{\textbf{D}}\textbf{\textit{u}})\textbf{\textit{n}}\big)_{\tau}=\nabla_{\tau}(\textbf{\textit{u}}\cdot\textbf{\textit{n}})+\left( \frac{\partial \textbf{\textit{u}}}{\partial \textbf{\textit{n}}}\right)_{\tau}-\sum_{k=1}^{2}\left(\textbf{\textit{u}}_{\tau}\cdot\frac{\partial\textbf{\textit{n}}}{\partial s_{k}} \right) \boldsymbol{\tau}_{k}.
\end{eqnarray*}
On the other hand, we have the following relation:
\begin{eqnarray*}\label{curl and déformation sur le bord}
\mathrm{curl}\,\textbf{\textit{u}}\times\textbf{\textit{n}}={-\nabla_{\tau}(\textbf{\textit{u}}\cdot\textbf{\textit{n}})+\left( \frac{\partial \textbf{\textit{u}}}{\partial \textbf{\textit{n}}}\right)_{\tau}+\sum_{k=1}^{2}\left(\textbf{\textit{u}}_{\tau}\cdot\frac{\partial\textbf{\textit{n}}}{\partial s_{k}} \right) \boldsymbol{\tau}_{k}}.
\end{eqnarray*}
In the particular case $\textbf{\textit{u}}\cdot\textbf{\textit{n}}=0$ on $\Gamma$, which implies that $$2\big((\mathrm{\textbf{D}}\textbf{\textit{u}})\textbf{\textit{n}}\big)_{\tau}={\mathrm{curl}\,\textbf{\textit{u}}\times\textbf{\textit{n}}-2\sum_{k=1}^{2}\left(\textbf{\textit{u}}_{\tau}\cdot\frac{\partial\textbf{\textit{n}}}{\partial s_{k}} \right) \boldsymbol{\tau}_{k}},$$   Comparing with~\eqref{Navier.BC}, the following boundary condition
\begin{eqnarray}\label{NSBC}
\textbf{\textit{u}}\cdot\textbf{\textit{n}}=0\,\,\,\,\text{and}\quad
\mathbf{curl}\,\textbf{\textit{u}}\times\textbf{\textit{n}}=\textbf{\textit{0}}\quad\text{ on }\,\Gamma,
\end{eqnarray}
is in fact a slip-without-friction Navier boundary condition type for more details the interested reader can also refer to~\cite{Veiga, Berselli, Mitrea-2009} and references therein.\\

\noindent Stokes problem with the boundary conditions~\eqref{NSBC} was studied by many authors. In the one hand, in the bounded domains, one can refer for instance to \cite{Medkova-2020,Nour_2011,Bramble-94,Conca-94,Conca-95}. In the other hand, the case for the exterior domains,  we can just mention~\cite{Meslamani_2013,LMR_2020}. The exterior problem~\eqref{Stokes}--\eqref{Navier.BC}, where as far as know, we can mention~\cite{Tartaglione-2020, Russo_JDE_2011, DMR-2019}. Although the Stokes problem set in bounded domains with conditions~\eqref{Navier.BC} has been well studied by various authors 
(see for instance~\cite{Ahmed_2014, Beirao_ADE_2004} or~\cite{Amrouche_M2AS_2016} for the case $\alpha=0$
and \cite{Amrita_2018,Ahmad_2014} for the case $\alpha\neq0$). \\

\noindent The purpose of this work is to study the existence and the uniqueness of the weak solution for the problem~\eqref{Stokes}--\eqref{NSBC}.
Since the flow domain is unbounded, we set the problem in weighted Sobolev spaces in order to control the behavior of functions at
infinity. This functional framework allows us to find solutions with different behaviours to infinity (decay or polynomial growth). The study is based on a $L^p$-theory, $p>3$. Our proof of solvability the problem~\eqref{Stokes}--\eqref{NSBC} is based on a variational formulations obtained and the Inf-Sup conditions.\\

\noindent The outline of this paper is as follows. In Section 2, we introduce the notations and the functional framework based on weighted Sobolev spaces. We shall precise in which sense, the Navier's type slip-without-friction boundary conditions~\eqref{NSBC},  we finish this section by giving a result concerning the Laplace problem with Neumann boundary conditions with data $f$ belongs to $L^{p}(\Omega)$, we got the existence of solutions in $W^{1,p}_{-1}(\Omega)$. In Section 3, we recall one result about the vector potential problem~ (see \cite{Louati_Meslameni_Razafison}) and we prove the Inf-Sup conditions, which plays a crucial role in the existence and uniqueness of solutions. Finally, in Section 4, we conclude with the main result of this paper related to well-posedness of the Stokes problem~\eqref{Stokes}--\eqref{NSBC}. We prove the existence and the uniqueness \textcolor{blue}{(up to an element of a nontrivial kernel)} of weak solutions, when $p>3$.

%%%%%%%%%%%%%%%%%%%%%%%%%%%%%%%%

\section{Notations and Preliminaries}
\subsection{Notations}
We recall the main notations and results, concerning the weighted Sobolev spaces, which we shall use later on.  In what follows, $p$ is a real number in the interval $]1,\infty[$. The dual exponent of $p$ denoted by $p'$ is given by the following relation: 
\begin{equation*}
 \frac{1}{p}+\frac{1}{p'}=1.
\end{equation*}

\noindent We will use bold characters for vector and matrix fields. A point in $\R^3$ is denoted by $\textbf{\textit{x}}=(x_1,x_2,x_3)$ and its
distance to the origin by
$$r=|\textbf{\textit{x}}|=\left(x_1^2+x_2^2+x_3^2\right)^{1/2}.$$

\noindent For any multi-index $\boldsymbol{\lambda}\in\N^3$, we denote
by $\partial^{\boldsymbol{\lambda}}$ the differential operator of order
$\boldsymbol{\lambda}$,
$$\partial^{\boldsymbol{\lambda}}=
{\partial^{|\lambda|}\over\partial_1^{\lambda_1}\partial_2^{\lambda_2}\partial_3^{\lambda_3}},
\quad|\lambda|=\lambda_1+\lambda_2+\lambda_3.$$ \noindent Let $\Omega'$  be a bounded connected open set in $\R^{3}$ with boundary $\partial \Omega'=\Gamma$ of class $\mathcal{C}^{1,1}$ and let $\Omega $ its complement \textit{i.e} $\Omega=\R^{3}\diagdown\overline{\Omega'}$. {We assume moreover that $\Omega$ is simply connected.} In this work, we shall also denote by $B_{R}$ the open ball of radius $R>0$ centred at the origin with boundary $\partial B_R$. In particular, since $\Omega'$ is bounded, we can find some $R_{0}$, such that $\Omega'\subset B_{R_{0}}$ and we introduce, for any $R\geq R_{0}$, the set
\begin{equation*}
\Omega_{R}=\Omega \cap B_{R}.
\end{equation*} 
\noindent We denote by $\mathcal{D}(\Omega)$ the space of
$\C^{\infty}$ functions with compact support in $\Omega$, $\mathcal{D}(\overline{\Omega})$ the restriction to $\Omega$ of functions belonging to 
$\mathcal{D}(\R^3)$. 
We recall that $\mathcal{D}'(\Omega)$ is the well-known space of distributions defined on
$\Omega$. We recall that $L^p(\Omega)$ is the well-known Lebesgue real space and for $m\ge1$, we recall that $W^{m,p}(\Omega)$ is the classical Sobolev space. 
We shall write $u\in
W_{loc}^{m,p}(\Omega)$ to mean that $u\in W^{m,p}(\mathcal{O})$, for any
bounded domain $\mathcal{O}$, with $\overline{\mathcal{O}}\subset\Omega$. We denote by $[s]$ the integer part of $s$. For any $k\in\Z$, $\mathcal{P}_k$ stands for the space of polynomials of degree less than or equal to $k$ and $\mathcal{P}^{\Delta}_{k}$ the harmonic polynomials of $\mathcal{P}_{k}$. If ${k}$ is a negative integer, we
set by convention $\mathcal{P}_{k}=\{0\}$.\\

\noindent Given a Banach space $X$, with dual space $X'$ and a closed subspace $Y$ of $X$, we denote by $X'\perp Y$ the subspace of $X'$ orthogonal to $Y$, i.e.
\begin{equation*}
X'\perp Y=\lbrace f \in X'; <f,v>=0\,\, \forall\, v \in Y\rbrace=(X/Y)'.
\end{equation*}
The space $X'\perp Y$ is also called the polar space of $Y$ in $X'$. Finally, as usual, $C>0$ denotes a generic constant the value
of which may change from line to line and even at the same line.\\

%%%%%%%%%%%%%%%%%%%%%%%%%%%%%

\subsection{Weighted Sobolev spaces}
\noindent In order to control the behavior at infinity of our functions and distributions we use for basic weights the quantity  $\rho(\textbf{\textit{x}})=(1+r^2)^{1/2}$ which is equivalent to $r$ at infinity, and to one on any bounded subset of $\R^3$.\\
For $\alpha\in\Z$, we introduce
$$W_{\alpha}^{0,p}(\Omega)=\Big\{u\in\mathcal{D}'(\Omega),\,\rho^{\alpha} u\in L^p(\Omega)\Big\},$$
which is a Banach space equipped with the norm: 
$$\|u\|_{W_{\alpha}^{0,p}(\Omega)}=\|\rho^{\alpha} u\|_{L^p(\Omega)}.$$ 

\noindent For any non-negative integers $m$, real numbers $p>1$ and $\alpha\in \Z$. We define the weighted Sobolev space for $3/p+\alpha\notin \{1,\cdots,m\}$:
$$
W_{\alpha}^{m,p}(\Omega)=\Big\{u\in \mathcal{D}'(\Omega);\,\forall\boldsymbol{\lambda}\in\N^{3}:
\,0\leq |\boldsymbol{\lambda}| \leq m,\,\rho^{\alpha-m+|\boldsymbol{\lambda}|}\partial^{\boldsymbol{\lambda}}u \in L^{p}(\Omega) \Big\}.
$$
It is a reflexive Banach space equipped with the norm:
$$
\|u\|_{W_{\alpha}^{m,p}(\Omega)}= \left(\sum_{0\leqslant|\boldsymbol{\lambda}|\leqslant m}
\|\rho^{\alpha-m+|\boldsymbol{\lambda}|}\partial^{\boldsymbol{\lambda}}u\|^{p}_{L^{p}(\Omega)}\right)^{1/p}.
$$
\noindent We define the semi-norm
$$|u|_{W_{\alpha}^{m,p}(\Omega)}=\left(\sum_{|\boldsymbol{\lambda}|=m}\|\rho^{\alpha}\partial^{\boldsymbol{\lambda}}u\|^{\textcolor{blue}{p}}_{L^p(\Omega)}\right)^{1/p}.$$

\noindent Let us give some examples of such space that will be often used in the remaining of the {paper}.
\begin{enumerate}

\item For $m=1$, we have
\begin{align*}
{W_{\alpha}^{1,p}(\Omega)}:={\{}&u \in \mathcal{D}'(\Omega);\,\rho^{\alpha-1}u \in L^{p}(\Omega),\,\rho^{\alpha}\,\nabla\,u \in L^{p}(\Omega) \}
\end{align*}

\item For $m=2$, we have

\begin{equation*}
W_{\alpha +1}^{2,p}(\Omega):=\left\lbrace u\in {W_{\alpha}^{1,p}(\Omega)}, \rho^{\alpha+1}\nabla^{2} u \in L^{p}(\Omega)\right\rbrace, 
\end{equation*}
\end{enumerate}

%\begin{rema}\label{remarkk}\quad\\
%\noindent Note that if\,\, $\dfrac{3}{p}+\alpha\in \left\lbrace 1,...,m\right\rbrace $, we need to add the logarithmic weight in the definition of weight Sobolev spaces introduced above. The logarithmic weight is defined by $ln(2+|x|^2)$ see~\cite{theseGiroire} for more details. For example if\,\, $3/p+\alpha=1$, we have
%\begin{align*}
%{{W}_{\alpha}^{1,p}(\Omega)}:={\{}&\textbf{\textit{u}} \in \mathcal{D}'(\Omega);\,\rho^{\alpha-1}(ln(2+r^2))^{-1}\textbf{\textit{u}} \in L^{p}(\Omega),\,\rho^{\alpha}\,\nabla\,\textbf{\textit{u}} \in L^{p}(\Omega) \}.
%\end{align*}
%\end{rema}

\noindent Now, we present some basic properties on weighted Sobolev spaces. For more details, the reader can refer to  \cite{Amrouche_JMPA_1997, Amrouche_1994, Hanouzet}. 
\begin{properties}\quad
\begin{enumerate}

\item The space $\mathcal{D}(\overline{\Omega})$ is dense in $W^{m,p}_{\alpha}(\Omega)$.\\
\item For any $m\in\N^*$ and $3/p+\alpha\neq 1$, we have the following continuous embedding:
\begin{equation}\label{inclusion.sobolev2}
W_{\alpha}^{m,p}(\Omega)\hookrightarrow W_{\alpha-1}^{m-1,p}(\Omega).
\end{equation}
\item For any $\alpha$, $m\in\Z$ and for any $\lambda \in \N^3$, the mapping
\begin{equation}\label{derive.espaces.poids}
u \in W_{\alpha}^{m,p}(\Omega)\longrightarrow\,\,\partial^{\lambda}u\in W_{\alpha}^{m-|\lambda|,p}(\Omega)
\end{equation}
is continuous.\\

\end{enumerate}
\end{properties}
\noindent The space $W_{\alpha}^{m,p}(\Omega)$  sometimes contains some polynomial functions. Let $j$ be defined as follow:
\begin{equation}
j=\begin{cases}
[m-(3/p+\alpha)] \quad\quad\mathrm{if}\quad 3/p+\alpha \notin \Z^{-},\\
m-3/p-\alpha-1\qquad\qquad \mathrm{otherwise}.
\end{cases}
 \end{equation}
 Then $\mathcal{P}_j$ is the space of all polynomials included in $W_{\alpha}^{m,p}(\Omega)$.

\noindent The norm of the quotient space $W_{\alpha}^{m,p}(\Omega)/\mathcal{P}_{j}$ is given by:
\begin{eqnarray*}
||u||_{W_{\alpha}^{m,p}(\Omega)/\mathcal{P}_{j}}=\inf_{\mu\in\mathcal{P}_{j}}||u+\mu ||_{W_{\alpha}^{m,p}(\Omega)}.
\end{eqnarray*}
%\end{definition}
All the local properties of $W_{\alpha}^{m,p}(\Omega)$ coincide with those of the corresponding classical Sobolev spaces $W^{m,p}(\Omega)$. Hence, it also satisfies the usual trace theorems on the boundary $\Gamma$. Therefore, we can define the space

$$\mathring{W}_{\alpha}^{m,p}(\Omega)=\lbrace u\in  W_{\alpha}^{m,p}(\Omega),\,\,\, \gamma_{0}u=0,\,\gamma_{1}u=0,\,\cdots, \gamma_{m-1}u=0\rbrace.$$

\noindent Note that when $\Omega=\R^{3}$, we have $\mathring{W}_{\alpha}^{m,p}(\R^{3})=W_{\alpha}^{m,p}(\R^{3})$. The space $\mathcal{D}(\Omega)$ is dense in $\mathring{W}^{m,p}_{\alpha}(\Omega)$. Therefore, the dual space of $\mathring{W}_{\alpha}^{m,p}(\Omega)$, denoting  by $W_{-\alpha}^{-m,p'}(\Omega)$, is a space of distributions with the norm
\begin{equation*}
||\,u\,||_{W_{-\alpha}^{-m,p'}(\Omega)}=\sup_{v \in \mathring{W}_{\alpha}^{m,p}(\Omega)}\dfrac{\left\langle u,v\right\rangle_{W_{-\alpha}^{-m,p'}(\Omega)\times \mathring{W}_{\alpha}^{m,p}(\Omega)} }{||\,v\,||_{W_{\alpha}^{m,p}(\Omega)}}.
\end{equation*}

\subsection{Functional spaces and Trace Results}
\noindent The purpose of this section is to introduce some weighted Sobolev spaces that are specific for the study of the Stokes problem~\eqref{Stokes} with the Navier boundary conditions~\eqref{NSBC}. Let us first introduce some notations. For any vector fields $\textit{\textbf{u}}$ and $\textit{\textbf{v}}$ of $\R^3$, we define 
$$\textit{\textbf{u}}\times\textit{\textbf{v}}=({u_2v_3-u_3v_2},\,u_3v_1-u_1v_3,\,u_1v_2-u_2v_1)^T$$
and $$\textbf{curl }\textit{\textbf{u}}=\nabla\times\textit{\textbf{u}}.$$

\noindent We note that, the vector-valued Laplace operator of a vector field $\textbf{\textit{v}}$ is equivalently defined by
\begin{equation}\label{delt1}
\Delta\,\textbf{\textit{v}}=\nabla\,\mathrm{div}\,\textbf{\textit{v}}-\mathbf{curl}\,\mathbf{curl}\,\textbf{\textit{v}}.
\end{equation}

\noindent We start by introducing the following spaces for $\alpha \in \Z$ and $1<p<\infty$.
\begin{defi}
The space $H_{\alpha}^{p}(\mathbf{curl},\Omega)$ is defined by:
\begin{equation*}
H_{\alpha}^{p}(\mathbf{curl},\Omega)=\left\lbrace \textbf{\textit{v}}\in W_{\alpha}^{0,p}(\Omega);\mathbf{curl}\,\, \textbf{\textit{v}}\in W_{\alpha+1}^{0,p}(\Omega)\right\rbrace \,,
\end{equation*}
and is provided with the norm:
$$ 
\|\textbf{\textit{v}}\|_{H_{\alpha}^{p}(\mathbf{curl},\Omega)}=\left( {\|\textbf{\textit{v}}\|^{p}}_{W_{\alpha}^{0,p}(\Omega)}+{\|\mathbf{curl}\,\, \textbf{\textit{v}}\|^{p}}_{W_{\alpha+1}^{0,p}(\Omega)}\right) ^{\frac{1}{p}}.
 $$
The space  $H_{\alpha}^{p}(\mathrm{div},\Omega)$ is defined by:
\begin{equation*}
H_{\alpha}^{p}(\mathrm{div},\Omega)=\left\lbrace \textbf{\textit{v}}\in W_{\alpha}^{0,p}(\Omega);\mathrm{div} \,\, \textbf{\textit{v}}\in W_{\alpha+1}^{0,p}(\Omega)\right\rbrace \,,
\end{equation*}
and is provided with the norm:
$$ 
\|\textbf{\textit{v}}\|_{H_{\alpha}^{p}(\mathrm{div},\Omega)}=\left( \|\textbf{\textit{v}}\|^{p}_{W_{\alpha}^{0,p}(\Omega)}+\|\mathrm{div} \,\, \textbf{\textit{v}}\|^{p}_{{W}_{\alpha+1}^{0,p}(\Omega)}\right) ^{\frac{1}{p}}.
 $$
Finally, we set 
\begin{equation*}
 \textbf{\textit{X}}_{\alpha}^{p}(\Omega)= H_{\alpha}^{p}(\mathbf{curl},\Omega)\cap H_{\alpha}^{p}(\mathrm{div},\Omega).
\end{equation*}
It is provided with the norm
$$ {\|\textbf{\textit{v}}\|_{\textbf{\textit{X}}_{\alpha}^{p}(\Omega)}}=\left( \|\textbf{\textit{v}}\|^{p}_{W_{\alpha}^{0,p}(\Omega)}+\|\mathrm{div} \,\, \textbf{\textit{v}}\|^{p}_{{W}_{\alpha+1}^{0,p}(\Omega)}+\|\mathbf{curl}\,\, \textbf{\textit{v}}\|^{p}_{W_{\alpha+1}^{0,p}(\Omega)}\right) ^{\frac{1}{p}}.$$
\noindent
These definitions will be also used when $\Omega$ is replaced by ${\R}^{3}$.
\end{defi}
\noindent
Note that $\mathcal{D}(\overline{\Omega})$ is dense in $H_{\alpha}^{p}(\mathrm{div},\Omega)$ and in $H_{\alpha}^{p}(\mathbf{curl},\Omega)$ and so in $\textbf{\textit{X}}_{\alpha}^{p}(\Omega)$. For the proof, one can use the same arguments than for the proof of the density of $\mathcal{D}(\overline{\Omega})$ in $W_{\alpha}^{m,p}(\Omega)$ (see \cite{theseGiroire, Hanouzet}). If $\textbf{\textit{v}}$ belongs to $H_{\alpha}^{p}(\mathrm{div},\Omega)$, then $\textbf{\textit{v}}$ has normal trace $\textbf{\textit{v}}\cdot \textbf{\textit{n}}$ in $W^{-1/p,p}(\Gamma)$, where $W^{-1/p,p}(\Gamma)$  denotes the dual space of $W^{1/p,p'}(\Gamma)$. By the same way, if $\textbf{\textit{v}}$ belongs to $H_{\alpha}^{p}(\mathbf{curl},\Omega)$, then $\textbf{\textit{v}}$ has a tangential trace $\textbf{\textit{v}}\times\textbf{\textit{n}}$ that belongs to  $W^{-1/p,p}(\Gamma)$. Similarly as in bounded domain, we have the trace theorems i.e, for $\alpha\in\Z$ there exists a constant $C>0$, such that
\begin{equation}\label{3.theorem de trace de H div}
 \forall \textbf{\textit{v}} \in H_{\alpha}^{p}(\mathrm{div},\Omega),\,\, ||\textbf{\textit{v}}\cdot \textbf{\textit{n}}||_{ W^{-1/p,p}(\Gamma)}\leq C||\textbf{\textit{v}}||_{H_{\alpha}^{p}(\mathrm{div},\Omega)},
\end{equation}
\begin{equation}\label{3.theorem de trace de H curl}
  \forall \textbf{\textit{v}} \in H_{\alpha}^{p}(\mathbf{curl},\Omega),\,\, ||\textbf{\textit{v}}\times \textbf{\textit{n}}||_{ W^{-1/p,p}(\Gamma)}\leq  C ||\textbf{\textit{v}}||_{H_{\alpha}^{p}(\mathbf{curl},\Omega)}
\end{equation}
and the following Green's formulas holds:
For any $\textbf{\textit{v}} \in H_{\alpha}^{p}(\mathrm{div},\Omega)$ and $\varphi \in W_{-\alpha}^{1,p'}(\Omega)$, we have
\begin{equation}\label{FG1}
\left\langle \textbf{\textit{v}}\cdot\textbf{\textit{n}}, \varphi\right\rangle_{\Gamma}=\int_{\Omega}\textbf{\textit{v}}\cdot\nabla\,\varphi\,dx+\int_{\Omega}\varphi\,\mathrm{div}\,\textbf{\textit{v}}\,dx.
\end{equation}

\noindent For any $\textbf{\textit{v}} \in H_{\alpha}^{p}(\mathbf{curl},\Omega)$ and $\boldsymbol{\varphi} \in W_{-\alpha}^{1,p'}(\Omega)$, we have
\begin{equation}\label{FG2}
\left\langle \textbf{\textit{v}}\times\textbf{\textit{n}}, \boldsymbol{\varphi}\right\rangle_{\Gamma}=\int_{\Omega}\textbf{\textit{v}}\cdot\mathbf{curl}\,\boldsymbol{\varphi}\,dx-\int_{\Omega}\mathbf{curl}\,\textbf{\textit{v}}\cdot\boldsymbol{\varphi}dx,
\end{equation}
where $\left\langle .,. \right\rangle_{\Gamma}$ denotes the duality pairing between $W^{-1/p,p}(\Gamma)$ and $W^{1/p,p'}(\Gamma)$.\\

\noindent The closures of $\boldsymbol{\mathcal{D}}(\Omega)$ in {$H_{\alpha}^{\,p}(\mathbf{curl},\Omega)$ and in $H_{\alpha}^{\,p}(\mathrm{div},\Omega)$} are denoted respectively by $\mathring{H}_{\alpha}^{\,p}(\mathbf{curl},\Omega)$ and $\mathring{H}_{\alpha}^{\,p}(\mbox{div},\Omega)$ and can be characterized respectively by:
$$
 \mathring{H}_{\alpha}^{\,p}(\mathbf{curl},\Omega)=\left\lbrace \textbf{\textit{v}}\in
 H_{\alpha}^{p}(\mathbf{curl},\Omega);\,\,\textbf{\textit{v}}\times \textbf{\textit{n}}=\textbf{0} \,\,\text{on}\,\, \Gamma\right\rbrace,
$$\vspace{.001cm}
$$
 \mathring{H}_{\alpha}^{\,p}(\mbox{div},\Omega)=\left\lbrace \textbf{\textit{v}}\in
 H_{\alpha}^{p}(\mbox{div},\Omega);\,\,\textbf{\textit{v}}\cdot\textbf{\textit{n}}=0\,\, \text{on}\,\,\Gamma\right\rbrace.
$$

\noindent For $1<p<\infty$ and $\alpha\in \Z$, we denote by $[\mathring{H}_{\alpha}^{p}(\mathrm{div};\Omega)]'$ and $[\mathring{H}_{\alpha}^{p}(\mathbf{curl};\Omega)]'$ the dual spaces of {$\mathring{H}_{\alpha}^{\,p}(\mbox{div},\Omega)$ and $\mathring{H}_{\alpha}^{\,p}(\mathbf{curl},\Omega)$} respectively. We can characterize theses spaces as it is stated in the following proposition.

\begin{prop}
\label{caracterisation du H'}\quad\\

 \begin{enumerate}
\item A distribution  $\textbf{\textit{f}}$ belongs to $[\mathring{H}_{\alpha}^{p}(\mathrm{div};\Omega)]'$ if and only if there exist functions $\boldsymbol{\psi}\in W^{0,p'}_{-\alpha}(\Omega)$ and $\chi\in W^{0,p'}_{-\alpha-1}(\Omega)$, 
such that $\textbf{\textit{f}}=\boldsymbol{\psi}+\nabla\chi$. Moreover
\begin{eqnarray}\label{estimation de Hdiv}
\|\boldsymbol{\psi}\|_{W^{0,p'}_{-\alpha}(\Omega)} + \|\chi\|_{W^{0,p'}_{-\alpha-1}(\Omega)}\leq C\|\textbf{\textit{f }}\|_{[\mathring{H}_{\alpha}^{p}(\mathrm{div};\Omega)]'}.
\end{eqnarray}
 
\item A distribution  $\textbf{\textit{f}}$ belongs to $[\mathring{H}_{\alpha}^{p}(\mathbf{curl},\Omega)]'$ if and only if there exist functions  $\psi\in W^{0,p'}_{-\alpha}(\Omega)$ and $\chi\in W^{0,p'}_{-\alpha-1}(\Omega)$, such that $\textbf{\textit{f}}=\psi+\mathbf{curl}\,\chi$. Moreover
\begin{eqnarray}
\Vert \psi\Vert_{W^{0,p'}_{-\alpha}(\Omega)} + \Vert \chi\Vert_{W^{0,p'}_{-\alpha-1}(\Omega)}\leq C \Vert \,\textbf{\textit{f}}\,\Vert_{[\mathring{H}_{\alpha}^{p}(\mathbf{curl},\Omega)]'}.
\end{eqnarray}
 
 \end{enumerate}
 \end{prop}

\noindent The proof of Proposition~\ref{caracterisation du H'} can be found in \cite[Proposition II.2]{dhifa2020}. As a consequence of Proposition~\ref{caracterisation du H'} and the imbedding~\eqref{inclusion.sobolev2} we have, for any $\alpha\in\Z$ and $1<p<\infty$, the following imbeddings
\begin{equation}\label{inclusion.Hdiv}
[\mathring{H}_\alpha^{p}(\mathrm{div};\Omega)]'\subset W_{-\alpha-1}^{-1,p'}(\Omega)
\end{equation} 
and
\begin{equation}\label{inclusion.Hcurl}
[\mathring{H}_\alpha^{p}(\mathrm{curl};\Omega)]'\subset W_{-\alpha-1}^{-1,p'}(\Omega).
\end{equation}

\noindent Let us consider the following space:
\begin{eqnarray*}
\textbf{E}^{p}(\Omega)=\lbrace \textbf{\textit{v}}\in W^{1,p}_{0}(\Omega); \Delta\, \textbf{\textit{v}}\in [\mathring{H}_{-1}^{p'}(\mathrm{div},\Omega)]'\rbrace.
\end{eqnarray*}

equipped with the following norm
\begin{eqnarray*}
\Vert \textbf{\textit{v}}\Vert_{\textbf{E}^{p}(\Omega)}=\Vert \textbf{\textit{v}}\Vert_{W^{1,p}_{0}(\Omega)}+\Vert \Delta \textbf{\textit{v}}\Vert_{[\mathring{H}_{-1}^{p'}(\mathrm{div},\Omega)]'}.
\end{eqnarray*}
We have the following preliminary result.
\begin{lemm}

 \textcolor{blue}{Let $1<p<\infty$ with $p\notin\{3/2,3\}$.} The space $\mathcal{D}(\overline{\Omega})$ is dense in $\textbf{E}^{p}(\Omega)$.\\
\end{lemm}
\begin{proof}
The proof is quite similar to the proof when p=2 (see~\cite{Meslamani_2013}). Let $P$ be a continuous linear mapping from $W_{0}^{1,p}(\Omega)$ to $W_{0}^{1,p}(\R^{3})$, such that $P\,\textbf{\textit{v}} |_{\Omega}=\textbf{\textit{v}}$ and let $\boldsymbol{\ell} \in (\textbf{\textit{E}}^{\,p}(\Omega))'$, such that for any $\textbf{\textit{v}} \in \boldsymbol{\mathcal{D}}(\overline{\Omega})$, we have $\left\langle \boldsymbol{\ell},\textbf{\textit{v}}\right\rangle =0$.
We want to prove that $\boldsymbol{\ell}=\boldsymbol{0}$ on $\textbf{\textit{E}}^{\,p}(\Omega)$. \textcolor{blue}{Since the mapping $\textbf{\textit{v}}\mapsto(\rho^{-1}\textbf{\textit{v}},\nabla\,\textbf{\textit{v}},\Delta\,\textbf{\textit{v}})$ is an isomorphism from $\textbf{\textit{E}}^{\,p}(\Omega)$ onto a closed subspace of $L^{p}(\Omega)\times L^{p}(\Omega)\times[\mathring{H}_{-1}^{\,p'}(\mathrm{div},\Omega)]'$, the Hahn-Banach theorem gives $\textbf{\textit{f}}_0,\,\textbf{\textit{F}}\in L^{p'}(\Omega)$ and $\textbf{\textit{g}}\in\mathring{H}_{-1}^{\,p'}(\mathrm{div},\,\Omega)$ such that, setting $\textbf{\textit{f}}=\rho^{-1}\widetilde{\textbf{\textit{f}}}_0-\mathrm{div}\,\widetilde{\textbf{\textit{F}}}\in W_{0}^{-1,p'}(\R^3)$, where $\widetilde{\textbf{\textit{f}}}_0$ and $\widetilde{\textbf{\textit{F}}}$ are the extensions by zero of $\textbf{\textit{f}}_0$ and $\textbf{\textit{F}}$,} for any $\textbf{\textit{v}}\in \textbf{\textit{E}}^{\,p}(\Omega)$,
\begin{equation*}
\left\langle \boldsymbol{\ell},\textbf{\textit{v}}\right\rangle =\left\langle  \textbf{\textit{f}},P\, \textbf{\textit{v}} \right\rangle _{W_{0}^{-1,p'}(\R^{3})\times W_{0}^{1,p}(\R^{3})}+\left\langle \Delta \textbf{\textit{v}},\textbf{\textit{g}}\right\rangle _{ [\mathring{H}_{-1}^{\,p'}(\mathrm{div},\,\Omega)]'\times \mathring{H}_{-1}^{\,p'}(\mathrm{div},\,\Omega)}.
\end{equation*}
{Note that $\left\langle\textbf{\textit{f}},\textbf{\textit{w}}\right\rangle=\int_{\Omega}(\textbf{\textit{f}}_0\cdot\rho^{-1}\textbf{\textit{w}}+\textbf{\textit{F}}:\nabla\,\textbf{\textit{w}})\,dx$ only depends on the restriction of $\textbf{\textit{w}}$ to $\Omega$.} Observe that we can easily extend by zero the function  $\textbf{\textit{g}}$ in such a way that $\widetilde{\textbf{\textit{g}}} \in H_{-1}^{\,p'}(\mathrm{div},\,\R^{3})$. Now we take $\boldsymbol{\varphi} \in \boldsymbol{\mathcal{D}}(\R^{3})$. Then we have by assumption that:
\begin{equation*}
\left\langle \textbf{\textit{f}},\boldsymbol{\varphi}\right\rangle  _{W_{0}^{-1,p'}(\R^{3})\times W_{0}^{1,p}(\R^{3})} + \displaystyle \int_{\R^{3}} \widetilde{\textbf{\textit{g}}} \cdot \Delta \boldsymbol{\varphi}dx =0,
\end{equation*}
because $\left\langle \textbf{\textit{f}},\boldsymbol{\varphi}\right\rangle =\left\langle  \textbf{\textit{f}},P\, \textbf{\textit{v}} \right\rangle $ where $\textbf{\textit{v}}=\boldsymbol{\varphi}|_{\Omega}$. Thus we have $\textbf{\textit{f}}+\Delta \widetilde{\textbf{\textit{g}}}=\boldsymbol{0}$ in $\boldsymbol{\mathcal{D}}'(\R^{3})$. Then we can deduce that  $\Delta \widetilde{\textbf{\textit{g}}}=-\textbf{\textit{f}} \in W_{0}^{-1,p'}(\R^{3})$ \textcolor{blue}{and, due to \cite[Theorem 1.3]{Amrouche_JMPA_1997}, the Laplace operator is an isomorphism from $W_{0}^{1,p'}(\R^{3})/\mathcal{P}^{\Delta}_{[1-3/p']}$ onto $W_{0}^{-1,p'}(\R^{3})\perp\mathcal{P}^{\Delta}_{[1-3/p]}$. If $p>3$, then $\mathcal{P}^{\Delta}_{[1-3/p]}=\R$ and, for every constant vector $\textbf{\textit{c}}$, which belongs to $W_{0}^{1,p}(\R^{3})$, a density argument gives $\langle\Delta\widetilde{\textbf{\textit{g}}},\textbf{\textit{c}}\rangle=\int_{\R^{3}}\widetilde{\textbf{\textit{g}}}\cdot\Delta\textbf{\textit{c}}\,dx=0$. Hence there exists $\boldsymbol{\lambda} \in W_{0}^{1,p'}(\R^{3})$ such that $\Delta \boldsymbol{\lambda}=\Delta \widetilde{\textbf{\textit{g}}}$. The harmonic function $\boldsymbol{\lambda}-\widetilde{\textbf{\textit{g}}}$ belongs to $W_{-1}^{0,p'}(\R^{3})$; it is therefore a constant vector, equal to zero if $p'<3$, and belonging to $W_{0}^{1,p'}(\R^{3})$ if $p'>3$. Thus $\widetilde{\textbf{\textit{g}}}\in W_{0}^{1,p'}(\R^{3})$.} Since $\widetilde{\textbf{\textit{g}}} \in W_{0}^{1,p'}(\R^{3})$ {vanishes in $\Omega'$}, we deduce that $\textbf{\textit{g}} \in \mathring{W}_{0}^{1,p'}(\Omega)$. As $\boldsymbol{\mathcal{D}}(\Omega)$ is dense in $\mathring{W}_{0}^{1,p'}(\Omega)$, there exists a sequence $\textbf{\textit{g}}_{k} \in \boldsymbol{\mathcal{D}}(\Omega)$ such that $\textbf{\textit{g}}_{k}\rightarrow\textbf{\textit{g}}$ in $W_{0}^{1,p'}(\Omega)$, when $k\rightarrow\infty$. Then $\nabla \cdot \textbf{\textit{g}}_{k}\rightarrow \nabla \cdot \textbf{\textit{g}}$ in $L^{\,p'}(\Omega)$. Since $W_{0}^{1,p'}(\Omega)$ is imbedded in $W_{-1}^{0,p'}(\Omega)$, we deduce that $\textbf{\textit{g}}_{k}\rightarrow\textbf{\textit{g}}$ in $H_{-1}^{\,p'}(\mathrm{div},\,\Omega)$.
Now, we consider $\textbf{\textit{v}} \in \textbf{\textit{E}}^{\,p}(\Omega)$ and we want to prove that $\left\langle  \boldsymbol{\ell},\textbf{\textit{v}}\right\rangle=0$.
Observe that:
\begin{equation*}
 \left\langle \boldsymbol{\ell},\textbf{\textit{v}}\right\rangle =-\left\langle \Delta \widetilde{\textbf{\textit{g}}},P \textbf{\textit{v}} \right\rangle _{W_{0}^{-1,p'}(\R^{3})\times W_{0}^{1,p}(\R^{3})} +\left\langle \Delta \textbf{\textit{v}},\textbf{\textit{g}}\right\rangle _{[\mathring{H}_{-1}^{\,p'}(\mathrm{div},\,\Omega)]'\times \mathring{H}_{-1}^{\,p'}(\mathrm{div},\,\Omega)}
\end{equation*}
\begin{equation*}
=\lim \limits_{k\rightarrow\infty}(-\displaystyle \int_{\Omega} \Delta \textbf{\textit{g}}_{k}\cdot \textbf{\textit{v}} dx+\left\langle \Delta \textbf{\textit{v}},\textbf{\textit{g}}_{k}\right\rangle _{[\mathring{H}_{-1}^{\,p'}(\mathrm{div},\,\Omega)]'\times \mathring{H}_{-1}^{\,p'}(\mathrm{div},\,\Omega)}
\end{equation*}
\begin{equation*}
=\lim \limits_{k\rightarrow\infty}(-\displaystyle \int_{\Omega} \Delta \textbf{\textit{g}}_{k}\cdot \textbf{\textit{v}} dx+\displaystyle \int_{\Omega} \textbf{\textit{v}} \cdot \Delta \textbf{\textit{g}}_{k} dx )=0.
\end{equation*}
\end{proof}
\noindent Next, we introduce the following space:
\begin{eqnarray*}
\textbf{\textit{V}}_{0,T}^{\,p}(\Omega)=\left\lbrace \textbf{\textit{z}} \in \textbf{\textit{X}}^{\,p}_{0,T}(\Omega);\,\,\mathrm{div}\,\textbf{\textit{z}}=0\,\,\mathrm{in}\,\,\,\Omega \right\rbrace,
\end{eqnarray*}
where $$\textbf{\textit{X}}^{\,p}_{0,T}(\Omega)=\lbrace \textbf{\textit{z}}\in {\textbf{\textit{X}}^{\,p}_{-1}(\Omega)}\,\,;\,\, \textbf{\textit{z}}\cdot\textbf{\textit{n}}=0\,\,\, on\,\,\,\Gamma\rbrace.$$
\noindent As a consequence, we have the following result.
\begin{lemm}
 The linear mapping $\gamma : \textbf{\textit{v}}\rightarrow \mathbf{curl}\,\textbf{\textit{v}}\mid_{\Gamma}\times \textbf{\textit{n}}$ defined on $\mathcal{D}(\overline{\Omega})$ can be extended to a linear continuous mapping
\begin{eqnarray*}
\gamma : \textbf{E}^{p}(\Omega)\longrightarrow W^{-1/p,p}(\Gamma).
\end{eqnarray*}
Moreover, we have the Green formula: for any $\textbf{\textit{v}}\in \textbf{E}^{p}(\Omega)$ and any $\varphi \in \textbf{\textit{V}}^{\,\, p'}_{0,T}(\Omega)$,
\begin{eqnarray}\label{formule de grenn dans E^p}
\langle -\Delta\,\textbf{\textit{v}},\varphi\rangle_{\Omega}=\int_{\Omega}\mathbf{curl}\,\textbf{\textit{v}}\,.\,\mathbf{curl}\,\varphi\,d\textbf{x}\textcolor{red}{-}\langle\mathbf{curl}\,\textbf{\textit{v}}\times\textbf{\textit{n}},\varphi\rangle_{\Gamma}
\end{eqnarray}
Where $\left\langle.\, ,\, .\right\rangle_{\Gamma}$ denotes the duality between $W^{-1/p,p}(\Gamma)$ and $W^{1/p,p'}(\Gamma)$ and $\left\langle.\, ,\, .\right\rangle_{\Omega}$ denotes the duality between $[\mathring{H}_{-1}^{\,\,p'}(\mathrm{div},\Omega)]'$ and $\mathring{H}^{\,\,p'}_{-1}(\mathrm{div},\Omega)$.
\end{lemm}
\begin{proof}

 Let $\textbf{\textit{v}} \in \boldsymbol{\mathcal{D}}(\overline{\Omega})$. Observe that if  $\boldsymbol{\varphi}\in W_{0}^{1,p'}(\Omega)$ such that {$\mathrm{div}\,\boldsymbol{\varphi}=0$ in $\Omega$ and} $\boldsymbol{\varphi}\cdot \textbf{\textit{n}}=0$ on $\Gamma$ we deduce that\\ $\boldsymbol{\varphi}\in {\textbf{\textit{V}}^{\,p'}_{0,T}(\Omega)}$, then \eqref{formule de grenn dans E^p} holds for such $\boldsymbol{\varphi} $.\\

\noindent  For every $\boldsymbol{\mu} \in W^{\,1/p,p'}(\Gamma)$, then there exists  $\boldsymbol{\varphi}\in W_{0}^{1,p'}(\Omega)$ such that $\boldsymbol{\varphi}=\boldsymbol{\mu}_{t}$ on $\Gamma$ and that $\mathrm{div}\,\boldsymbol{\varphi}=0$ with 
\begin{equation}\label{estimation de relevement n1}
||\boldsymbol{\varphi}||_{W_{0}^{1,p'}(\Omega)}\leqslant C ||\boldsymbol{\mu}_{t}||_{W^{\,1/p,p'}(\Gamma)}\leqslant C ||\boldsymbol{\mu}||_{W^{\,1/p,p'}(\Gamma)}.
\end{equation}
As a consequence, using \eqref{formule de grenn dans E^p} we have
\begin{equation*}
|\langle \mathbf{curl}\,\textbf{\textit{v}}\times\textbf{\textit{n}},\,\boldsymbol{\mu}\rangle_{\Gamma}|\leqslant C||\textbf{\textit{v}}||_{\textbf{\textit{E}}^{\,p}(\Omega)}||\boldsymbol{\mu}||_{W^{\,1/p,p'}(\Gamma)}.
\end{equation*}
Thus,
\begin{equation*}
|| \mathbf{curl}\,\textbf{\textit{v}}\times\textbf{\textit{n}}||_{W^{\,-1/p,p}(\Gamma)}\leqslant C||\textbf{\textit{v}}||_{\textbf{\textit{E}}^{\,p}(\Omega)}.
\end{equation*}
We deduce that the linear mapping $\gamma$ is continuous for the norm $\textbf{\textit{E}}^{\,p}(\Omega)$. Since $\boldsymbol{\mathcal{D}}(\overline{\Omega})$ is dense in $\textbf{\textit{E}}^{\,p}(\Omega)$, $\gamma$ can be extended to by continuity to $\gamma \in \mathcal{L}(\textbf{\textit{E}}^{\,p}(\Omega), W^{\,-1/p,p}(\Gamma))$ and formula \eqref{formule de grenn dans E^p} holds for all $\textbf{\textit{v}}\in\textbf{\textit{E}}^{\,p}(\Omega)$ and $\boldsymbol{\varphi}\in W_{0}^{1,p'}(\Omega)$ such that $\mathrm{div}\,\boldsymbol{\varphi}=0$ in $\Omega$ and $\boldsymbol{\varphi}\cdot \textbf{\textit{n}}=0$ on $\Gamma$.
\end{proof}
%%%%%%%%%%%%%%%%%%%%%%%%%%%%%%%
\subsection{The Laplace problem}
This section is devoted to recall the solution of the Laplace equations in $\R^3$ and we give a result concerning the Laplace problem with Neumann
boundary conditions. The result that we state below was proved by Amrouche et {al.} in~\cite{Amrouche_1994}.
\begin{theo}\label{isomorphism p>3}
Assume that $3/p+\alpha\neq 1$ and $3/p'-\alpha\neq 1$, then for any integer $m \geqslant 1$,  the following Laplace operator is an isomorphism:
\begin{eqnarray}\label{laplce pour m}
\Delta : W^{1+m,p}_{\alpha +m}(\R^3)/P_{[1-3/p-\alpha]}^{\bigtriangleup}  & \longrightarrow & W^{-1+m,p}_{\alpha +m}(\R^3)\perp P_{[1-3/p'+\alpha]}^{\bigtriangleup}
\end{eqnarray}
\end{theo}

\noindent Now, we are interested into the following Neumann problem:
\begin{eqnarray}\label{Problem de Neumann Ax}
-\Delta\,u=f\quad\text{ in}\,\,\,\Omega\quad\text{ and }\quad
\dfrac{\partial u}{\partial n}=g\quad\text{ on}\,\,\,\Gamma.
\end{eqnarray}
  J. Giroire in~\cite{theseGiroire}, studied the problem~\eqref{Problem de Neumann Ax} in the Hilbert framework. In~\cite{Louati_Meslameni_Razafison}, the authors investigated the harmonic Neumann problem in $L^p$ theory, for the exterior domain with boundary of class $\mathcal{C}^{1,1}$, note that these results have been also proved by Specovius-Neugebauer \cite{specovius_1990} with boundary of class at \textcolor{blue}{least} $\mathcal{C}^2$. C. Amrouche, V. Girault and J. Giroire in~\cite{Amrouche_JMPA_1997}, studied the problem~\eqref{Problem de Neumann Ax} with data f belongs to $W^{-1,p}_{0}(\Omega)\cap L^p(\Omega)$, they got the existence of solutions in $W^{1,p}_{0}(\Omega)$. In this work, we give a result concerning the Laplace problem with Neumann
boundary conditions with data $f$ belongs to $L^{p}(\Omega)$, we got the existence of solutions in $W^{1,p}_{-1}(\Omega)$.\\

%\noindent Our first proposition is established also in \cite{Louati_Meslameni_Razafison}, it characterizes the kernel of the Laplace operator with Neumann boundary condition.
\noindent \textcolor{blue}{ Let us first introduce the kernel of the Laplace operator with Neumann boundary condition. For any integer $\alpha\in \Z$ and $1<p<\infty$}
\begin{equation*}
\mathcal{N}_{p,\alpha}^{\Delta}=\left\lbrace w\in W_{\alpha}^{1,p}(\Omega);\,\,\,\,\Delta\,w=0\quad\mathrm{in}\,\, \Omega\quad\mathrm{and}\quad \dfrac{\partial w}{\partial n}= 0\quad\mathrm{on}\,\, \Gamma\right\rbrace.
\end{equation*}
%\begin{prop}\label{noyau de neumann}
%For any integer $\alpha\geq1$, $\mathcal{N}_{p,\alpha}^{\Delta}$ the subspace of all functions in $W_{\alpha}^{1,p}(\Omega)$ of the form $w(p)-p$, where $p$ runs over all polynomials of $\mathcal{P}_{[1-3/p-\alpha]}^{\Delta}$ and $w(p)$ is the unique solution in $W_{0}^{1,2}(\Omega)\cap W_{\alpha}^{1,p}(\Omega)$ of the Neumann problem
%\begin{equation}\label{problem de Neumann liee au noyeau}
%\Delta\,w(p)=0\quad\mathrm{in}\,\, \Omega\quad\mathrm{and}\quad\dfrac{\partial w(p)}{\partial  n}=\dfrac{\partial p}{\partial  n}\quad\mathrm{on}\,\, \Gamma.
%\end{equation}
%Here also, we set $\mathcal{N}_{p,\alpha}^{\Delta}=\left\lbrace 0\right\rbrace $ when $\alpha\leq0$; $\mathcal{N}_{p,\alpha}^{\Delta}$ is a finite-dimentional space of the same dimension as $\mathcal{P}_{[1-3/p-\alpha]}^{\Delta}$ and in particular, $\mathcal{N}_{0}^{\Delta}=\R$.
%\end{prop}
\noindent The next theorem states an existence, uniqueness and regularity result for problem~\eqref{Problem de Neumann Ax}.
\begin{theo}\label{theorem Neumann}
Let $1<p<\infty$\textcolor{blue}{, $p\notin\{3/2,3\}$,} be any real number and assume that $\Gamma$ is of class $\mathcal{C}^{1,1}$ if $p\neq 2$.
For any $f$ in $L^{p}(\Omega)$ and $g$ in $W^{-1/p,p}(\Gamma)$. Then, the problem~\eqref{Problem de Neumann Ax} has a solution $u\in W^{1,p}_{-1}(\Omega)$  unique up to element of $\mathcal{N}^{\Delta}_{p,-1}(\Omega)$ and we have the following estimate: 
\begin{eqnarray}\label{estimation faible neumann}
\Vert u\Vert_{W^{1,p}_{-1}(\Omega)/ \mathcal{N}^{\Delta}_{p,-1}(\Omega)} \leqslant C\big(\Vert f\Vert_{L^{p}(\Omega)}+\Vert g\Vert_{W^{-1/p,p}(\textcolor{blue}{\Gamma})}\big)
\end{eqnarray}
If in addition, \textcolor{blue}{$g$ belongs to $W^{1/p',p}(\Gamma)$}, the solution $u$ of problem~\eqref{Problem de Neumann Ax} belongs to $W^{2,p}_{0}(\Omega)$ and satisfies 
\begin{eqnarray}\label{estimation forte neumann}
\Vert u\Vert_{W^{2,p}_{0}(\Omega)\textcolor{blue}{/ \mathcal{N}^{\Delta}_{p,-1}(\Omega)}} \leqslant C\big(\Vert f\Vert_{L^{p}(\Omega)}+\Vert g\Vert_{W^{1/p',p}(\textcolor{blue}{\Gamma})}\big)
\end{eqnarray}
\end{theo}
%%%%%%%%%%%%%%%%%%%%%%%%%%%%%%%%%%%%%%
\begin{proof}
Let us extend $f$ by zero in $\Omega'$ and let $\widetilde{f}$ denote the extended function. Then $\widetilde{f}$ belongs to $L^{p}(\R^3)$. Using Theorem~\ref{isomorphism p>3},  there exists a unique function $\widetilde{v} \in W^{2,p}_{0}(\R^3)/\mathcal{P}^{\Delta}_{[2-3/p]}$ such that
\begin{eqnarray}\label{laplace in R^3}
-\Delta\,\widetilde{v}=\widetilde{f}\quad\text{ in }\,\,\R^3.
\end{eqnarray}
Then $\nabla\,\widetilde{v}\cdot n$ belongs to $W^{1-1/p,p}(\Gamma)\hookrightarrow W^{-1/p,p}(\Gamma)$. It follows from \cite[Theorem 3.12]{Louati_Meslameni_Razafison}, that the following problem:
\begin{eqnarray}\label{Neumann Harmonic}
\Delta\,w=0\quad\text{ in }\,\,\,\Omega\quad\text{ and } \quad\nabla\,w\cdot n=g-\nabla\,\widetilde{v}\cdot n\quad\text{ on }\,\,\,\Gamma,
\end{eqnarray}
has a solution $w\in W^{1,p}_{-1}(\Omega)$ unique up to element of $\mathcal{N}^{\Delta}_{p,-1}(\Omega)$. Thus $u=\widetilde{v}_{\mid_{\Omega}}
+w\in W^{1,p}_{-1}(\Omega)$ is the required solution of~\eqref{Problem de Neumann Ax}. \textcolor{blue}{The uniqueness follows immediately from \cite[Proposition~3.11]{Louati_Meslameni_Razafison}}.\\

\noindent Now, suppose that $g$ belongs to $W^{1/p',p}(\Gamma).$ The aim is to  prove that $u$ belongs to $W^{2,p}_{0}(\Omega)$. To that end, let us introduce the following partition of unity:
\begin{equation*}
 \label{partition de l'unite}
 \begin{split}
&\varphi,\,\psi\in\mathcal{C}^\infty(\R^3), \quad 0\le\varphi,\,\psi\le1,\quad \varphi+\psi=1\quad\text{in}\quad\R^3,\\ 
&\varphi=1\quad\text{in}\quad B_R,\quad\text{supp }\varphi\subset B_{R+1}.
\end{split}
\end{equation*}
Let $P$ be a continuous linear mapping from $W^{1,p}_{-1}(\Omega)$ to $W^{1,p}_{-1}(\R^3)$, such that $P\,u=\widetilde{u}$. Then $\widetilde{u}$ belongs to $W^{1,p}_{-1}(\R^3)$ and can be written as:
\begin{eqnarray*}
\widetilde{u}=\varphi\,{\widetilde{u}}+\psi\,\widetilde{u}.
\end{eqnarray*}
Next, one can easily observe that $\widetilde{u}$ satisfies the following problem:
\begin{eqnarray}\label{laplace R^3}
-\Delta\,\psi\,\widetilde{u}=f_{1}\quad\text{ in }\,\,\,\R^3,
\end{eqnarray}
with
\begin{eqnarray*}
f_{1}=\widetilde{f}{\psi}
-(2\nabla{\widetilde{u}}\nabla{\psi}+\widetilde{u}\Delta\,{\psi}).
\end{eqnarray*}
Owing to the support of $\psi$, $f_{1}$ has the same regularity as $f$ and so belongs to $L^{p}(\R^3)$. It follows from Theorem~\ref{isomorphism p>3}, that there exists $z$ in $W^{2,p}_{0}(\R^3)$ such that $-\Delta\,z=f_{1}$ in $\R^3.$ This implies that $\psi\,{\widetilde{u}}-z$ is a harmonic tempered distribution and therefore a harmonic polynomial that belongs to $\mathcal{P}^{\Delta}_{[2-3/p]}$. The fact that $\mathcal{P}^{\Delta}_{[2-3/p]} \subset  W^{2,p}_{0}(\R^3)$ yields that ${\psi\,\widetilde{u}}$ belongs to $W^{2,p}_{0}(\R^3)$. In particular, we have ${\psi\,\widetilde{u}}=u$ outside $B_{R+1}$, so the restriction of $u$ to $\partial\, B_{R+1}$ belongs to $W^{2-1/p,p}(\partial\,B_{R+1})$. Therefore, $\varphi\,u\in W^{1,p}(\Omega_{R+1})$ satisfies the following problem:
\begin{eqnarray}\label{mixte laplace}
\begin{cases}
-\Delta\,\varphi\,u=f_{2}\quad\text{ in }\,\,\, \Omega_{R+1},\\
\nabla\,\varphi\,u\cdot n=g\quad\text{ on }\,\,\,\Gamma,\\
\varphi\,u={0}\quad\text{ on }\,\,\,\partial\, B_{R+1},
\end{cases}
\end{eqnarray}
where $f_{2}\in L^{p}(\Omega_{R+1})$ have similar expression as  $f_{1}$ with $\psi$ {replaced} by $\varphi$. According [Remark 3.2, \cite{Amrouche_JMPA_1997}], $\varphi\,u\in W^{2,p}(\Omega_{R+1})$, which in turn shows that $\varphi\,\widetilde{u}$ also belongs to $W^{2,p}(\Omega_{R+1})$. This implies that $u\in W^{2,p}_{0}(\Omega).$\\
\end{proof}
%%%%%%%%%%%%%%%%%%%%%%%%%%%%%%%%%%%
\section{Potential vector and inf-sup condition}
\noindent In this section, we recall a result concerning the existence a vector potential, and we give the proof of the Inf-sup condition, which plays a crucial role in the existence and uniqueness of solutions. Let us recall the abstract setting of Babu$\check{s}$ka-Brezzi's Theorem (see  Babu$\check{s}$ka $\cite{Babuska}$ and Brezzi $\cite{Brezzi}$).
\begin{theo}\label{Babuchca brezi}
 Let $X$ and $M$ be two reflexive Banach spaces and $X'$ and $M'$ their dual spaces. Let $a$ be the continuous bilinear form defined on $X\times M$, let $A\in \mathcal{L}(X;\ M')$ and  $A'\in \mathcal{L}(M;\ X')$ be the operators defined by
$$
\forall v\in X,\ \forall w\in M, \ a(v, w) =\ <Av, w>\ =\ <v, A'w>
$$
and $V=\mathrm{Ker}\,A$. 
The following statements are equivalent:\medskip

\noindent i) There exist $\beta > 0$ such that
\begin{equation}\label{CISabstraite}
\inf_{\substack{w\in M\\ w\not=0}}\,\sup_{\substack{v\in X\\ v\not=0}}\dfrac{a(v, w)}{\Vert v\Vert_X\,\Vert w\Vert_M}\geq\beta.
\end{equation}\medskip
\noindent ii) The operator $A:\ X/V\mapsto M'$ is an {isomorphism} and $1/\beta$ is the continuity constant of $A^{-1}$.\medskip

\noindent iii)  The operator $A':\ M\mapsto X'\bot \,V$ is an {isomorphism} and $1/\beta$ is the continuity constant of $(A')^{-1}$.

\end{theo}\medskip

\begin{rema}\label{remark sur c inf sup}{\rm

As consequence, if the Inf-Sup condition (\ref{CISabstraite}) is satisfied, then we have the following properties:\medskip

\noindent i) If $V = \{0\}$, then for any $f\in X'$, there exists a unique $ w\in M$ such that
\begin{equation}\label{inegalite util ds inf-sup}
\forall v\in X,\  a(v, w) = \ <f, v>\quad \mathrm{and}\quad \Vert w\Vert_{M} \leq \dfrac{1}{\beta} \Vert f\Vert_{X'}.
\end{equation}

\noindent ii) If $V \not= \{0\}$, then for any $f\in X',\ $ satisfying the compatibility condition: \medskip

$\forall v\in V,\ <f, v>\ = 0$, there exists a unique $ w\in M$ such that \eqref{inegalite util ds inf-sup}.\medskip

\noindent iii) For any $g\in M',\ \exists v\in X$, unique up an additive element of $V$, such that: 
$$
\forall w\in M,\  a(v, w) = \ <g, w> \quad \mathrm{and}\quad \Vert v\Vert_{X/V} \leq \dfrac{1}{\beta} \Vert g\Vert_{M'}.
$$
}
\end{rema}
\noindent Now, we recall a result concerning the potential vector provided in~\cite{Louati_Meslameni_Razafison}.
\begin{theo}\label{Theorem Vecteur potentiel}
Assume that $1<p<3/2$. Let $\textbf{\textit{z}}$ in $H^{p}_{0}(\mathrm{div},\Omega)$  satisfying
\begin{center}
$\mathrm{div}\,\textbf{\textit{z}}=0$ \text{ in } $\Omega$\quad\text{ and }\quad $\langle\textbf{\textit{z}}\cdot\textbf{\textit{n}},1\rangle_{\Gamma }=0$.
\end{center}
 Then there exists a unique vector potential $\boldsymbol{\psi}$ in $W^{1,p}_{0}(\Omega)$ such that:
\begin{eqnarray}\label{vecteur potentiel p}
\textbf{\textit{z}}=\textbf{curl}\,\boldsymbol{\psi}\text{,} \quad\mathrm{div}\,\boldsymbol{\psi}=0\quad \text{ in } \Omega\quad\text{ and }\quad
\boldsymbol{\psi}\cdot\textbf{\textit{n}}=0\quad\text{ on }\,\,\Gamma.
\end{eqnarray}
In addition  we have the following estimate:
\begin{eqnarray}\label{curl inégalité}
\Vert\boldsymbol{\psi}\Vert_{W^{1,p}_{0}(\Omega)}\leqslant C\Vert\textbf{\textit{z}}\Vert_{L^p(\Omega)}.
\end{eqnarray}
\end{theo}
\noindent We introduce the following space:
\begin{eqnarray*}
\textbf{\textit{Z}}^{\,p}_{T}(\Omega)=\Big\lbrace \textbf{\textit{v}}\in W^{0,p}_{-1}(\Omega):\,\mathrm{div}\,\textbf{\textit{v}}\in L^{p}(\Omega),\,\,\mathbf{curl}\,\textbf{\textit{v}}\in L^{p}(\Omega)\,\, ,\,\,\textbf{\textit{v}}\cdot\textbf{\textit{n}}=0\,\,\text{ on }\,\,\Gamma\Big\rbrace.
\end{eqnarray*} 
\noindent\textcolor{blue}{This space is equipped with the norm of $\textbf{\textit{X}}^{\,p}_{-1}(\Omega)$, i.e. $\Vert \textbf{\textit{v}}\Vert_{\textbf{\textit{Z}}^{\,p}_{T}(\Omega)}=\big(\Vert \textbf{\textit{v}}\Vert^{p}_{W^{0,p}_{-1}(\Omega)}+\Vert \mathrm{div}\,\textbf{\textit{v}}\Vert^{p}_{L^{p}(\Omega)}+\Vert \mathbf{curl}\,\textbf{\textit{v}}\Vert^{p}_{L^{p}(\Omega)}\big)^{1/p}$, and we have $\textbf{\textit{V}}^{\,p}_{0,T}(\Omega)=\lbrace \textbf{\textit{v}}\in \textbf{\textit{Z}}^{\,p}_{T}(\Omega);\ \mathrm{div}\,\textbf{\textit{v}}=0\rbrace$.}
\begin{theo}\label{injection in X_T  homogene}
Let $1<p<3/2$. There exists a constant $C >0$ such that for all $\boldsymbol{\varphi} \in \textbf{\textit{Z}}^{\,p}_{T}(\Omega),$ we have
\begin{eqnarray}\label{equivalence de norme avec curl X_T p}
||\boldsymbol{\varphi}||_{W_{0}^{1,p}(\Omega)}&\leq& C \Big(||\mathrm{div}\,\boldsymbol{\varphi}||_{L_{}^{p}(\Omega)}+||\mathbf{curl}\,\boldsymbol{\varphi} ||_{L_{}^{p}(\Omega)}\Big).
\end{eqnarray}
\end{theo}
\begin{proof}
Let $\boldsymbol{\varphi}$ in $\textbf{\textit{Z}}^{\,p}_{T}(\Omega)$. It follows from [\cite{Louati_Meslameni_Razafison}, Theorem 5.1], that $\textbf{\textit{Z}}^{\,p}_{T}(\Omega)$ is continuously imbedded in $W^{1,p}_{0}(\Omega)$ and so $\boldsymbol{\varphi}$ belongs to $W^{1,p}_{0}(\Omega)$. Observe that for any $\chi \in W_{0}^{1,p'}(\Omega)$, the following Green formula holds:
\begin{eqnarray}\label{fG valide pour 1}
\langle\mathbf{curl}\,\varphi\cdot\textbf{\textit{n}},\chi \rangle_{\Gamma}=\displaystyle \int_{\Omega} \mathbf{curl}\,\varphi\cdot \nabla\,\chi\, d\textbf{x}.
\end{eqnarray}
Because $1<p<3/2$, the constants belong to $W_{0}^{1,p'}(\Omega)$, then \eqref{fG valide pour 1} is still valid for $\chi=1$ and thus we have:
$$\langle\mathbf{curl}\,\varphi\cdot\textbf{\textit{n}},1\rangle_{\Gamma}=0.$$
Using Theorem \ref{Theorem Vecteur potentiel}, we prove that there exists a unique vector potential $\boldsymbol{\psi} \in W_{0}^{1,p}(\Omega)$ such that:
\begin{eqnarray*}
\mathbf{curl}\,\boldsymbol{\varphi}=\textbf{curl}\,\boldsymbol{\psi}\text{,} \quad\mathrm{div}\,\boldsymbol{\psi}=0\quad \text{ in } \Omega\quad\text{ and }\quad
\boldsymbol{\psi}\cdot\textbf{\textit{n}}=0\quad\text{ on }\,\,\Gamma.
\end{eqnarray*}
In addition, we have
\begin{eqnarray*}
\Vert\boldsymbol{\psi}\Vert_{W^{1,p}_{0}(\Omega)}\leqslant C\Vert \mathbf{curl}\,\boldsymbol{\varphi}\Vert_{L^p(\Omega)}.
\end{eqnarray*}
On the other hand, let us solve the following problem,
\begin{eqnarray}\label{Neuman problem}
\Delta\,\w=\mathrm{div}\,\boldsymbol{\varphi}\,\,\text{ in }\,\,\Omega\quad\text{ and }\,\,\dfrac{\partial\w}{\partial\textbf{\textit{n}}}=0\,\,\text{ on }\,\,\Gamma.
\end{eqnarray}
Thanks to Theorem \ref{theorem Neumann}, problem \eqref{Neuman problem} has a solution $\w \in W^{2,p}_{0}(\Omega)$ that satisfies
\begin{eqnarray}\label{inégalité divergence}
\Vert\nabla\,\w\Vert_{W^{1,p}_{0}(\Omega)}\leqslant C\Vert\mathrm{div}\,\boldsymbol{\varphi}\Vert_{L^{p}(\Omega)}
\end{eqnarray}
Setting $\z=\boldsymbol{\varphi}-\nabla\,\w$. It is clear that $\z$ belongs to $W^{1,p}_{0}(\Omega)$, it is divergence-free,  $\z\cdot\textbf{\textit{n}}=0$ on $\Gamma$ and $\mathbf{curl}\,\varphi=\mathbf{curl}\,\z$ in $\Omega$. The uniqueness of $\boldsymbol{\psi}$ implies that $\z=\boldsymbol{\psi}$ and thus we obtain
\begin{eqnarray}\label{inégalité curl}
\Vert\z\Vert_{W^{1,p}_{0}(\Omega)}\leqslant C\Vert \mathbf{curl}\,\boldsymbol{\varphi}\Vert_{L^p(\Omega)}.
\end{eqnarray}
Finally, the inequality \eqref{equivalence de norme avec curl X_T p} follows immediately from \eqref{inégalité curl} and \eqref{inégalité divergence}.
\end{proof}

\begin{rema}\label{equiavlence des normes}
As a consequence of Theorem \ref{injection in X_T  homogene}, \textcolor{blue}{for $1<p<3/2$,} the seminorm in the right-hand side of \eqref{equivalence de norme avec curl X_T p} is a norm on $\textbf{\textit{Z}}^{\,p}_{T}(\Omega)$ equivalent to the norm $||\boldsymbol{\varphi}||_{W_{0}^{1,p}(\Omega)}$.
\end{rema}

\noindent The "inf-sup" condition is given by the following lemma:
\begin{lemm}\label{lemme condition inf sup 2}
Assume that $p>3$. The following  Inf-Sup Condition holds: there exists a constant $\beta>0$, such that
\begin{equation}\label{condition inf sup 2}
 \inf_{\substack{\boldsymbol{\varphi}\in\textbf{\textit{V}}^{\,p'}_{0,T}(\Omega)\\\boldsymbol{\varphi}\neq 0}}\sup_{\substack{\boldsymbol{\psi}\in\textbf{\textit{V}}^{\,p}_{0,T}(\Omega)\\\boldsymbol{\psi}\neq 0}}\dfrac{\int_{\Omega}\mathbf{curl}\,\boldsymbol{\psi}\cdot\mathbf{curl}\,\boldsymbol{\varphi}\,\mathrm{d}\textbf{\textit{x}}}{\Vert \boldsymbol{\psi}\Vert_{\textbf{\textit{Z}}^{\,p}_{T}(\Omega)}\Vert \boldsymbol{\varphi}\Vert_{\textbf{\textit{Z}}^{\,p'}_{T}(\Omega)}}\geq\beta.
\end{equation}
\end{lemm}

%%%%%%%%%%%%%%%%%%%%

\begin{proof}
Let $\textbf{\textit{g}}\in\textbf{\textit{L}}_{}^{p}(\Omega)$ and let us introduce the following Dirichlet problem: 
\begin{equation*}
 {\Delta\,\chi}=\mathrm{div}\,\textbf{\textit{g}}\quad\mathrm{in}\,\, \Omega,\quad\chi=0\quad\mathrm{on}\,\, \Gamma. 
\end{equation*}
It is shown in Theorem $3.7$ of \cite{Louati_Meslameni_Razafison}, that this problem has a solution $\chi \in \mathring{W}_{0}^{1,p}(\Omega) $ and we have
\begin{equation*}
 \Vert\nabla\,\chi \Vert_{\textbf{\textit{L}}_{}^{p}(\Omega)}\leq C\Vert\textbf{\textit{g}}\Vert_{\textbf{\textit{L}}_{}^{p}(\Omega)}.
\end{equation*}
Set $\textbf{\textit{z}}= \textbf{\textit{g}}-\nabla\,\chi$. Then we have $\textbf{\textit{z}} \in \textbf{\textit{L}}_{}^{p}(\Omega)$, $\mathrm{div}\,\textbf{\textit{z}}=0$ and we have 
\begin{equation} \label{inegalite decomposition de helmholt 2}
 \Vert\textbf{\textit{z}} \Vert_{\textbf{\textit{L}}_{}^{p}(\Omega)}\leq C\Vert\textbf{\textit{g}}\Vert_{\textbf{\textit{L}}_{}^{p}(\Omega)}.
\end{equation}
Let $\boldsymbol{\varphi}$ any function of $\textbf{\textit{V}}^{\,p'}_{0,T}(\Omega)$, by Theorem \ref{injection in X_T  homogene} we have $\boldsymbol{\varphi} \in \textbf{\textit{Z}}^{\,p'}_{T}(\Omega) \hookrightarrow W_{0}^{1,p'}(\Omega)$. Then due to \eqref{equivalence de norme avec curl X_T p} we can write
\begin{equation}\label{dualite 2}
 \Vert\boldsymbol{\varphi}\Vert_{W^{\,1,p'}_{0}(\Omega)}\leq C\Vert\mathbf{curl}\,\boldsymbol{\varphi}\Vert_{\textbf{\textit{L}}_{}^{p'}(\Omega)}= C\,\sup_{ \substack{\textbf{\textit{g}}\in\textbf{\textit{L}}_{}^{p}(\Omega)\\\textbf{\textit{g}}\neq 0}}\dfrac{\big{\vert}\int_{\Omega}\mathbf{curl}\,\boldsymbol{\varphi}\cdot\textbf{\textit{g}}\,\mathrm{d}\textbf{\textit{x}}\big{\vert}}{\Vert \textbf{\textit{g}}\Vert_{\textbf{\textit{L}}_{}^{p}(\Omega)}}.
\end{equation}
Using the fact that $\mathbf{curl}\,\boldsymbol{\varphi} \in H_{0}^{\,p'}(\mathrm{div},\Omega)$ and applying \eqref{FG1}, we obtain
\begin{equation}\label{calcul integ zero 2}
 \int_{\Omega}\mathbf{curl}\,\boldsymbol{\varphi}\cdot\nabla\,\chi\,\mathrm{d}\textbf{\textit{x}}=0.
\end{equation}
Now, let $\lambda \in W_{1}^{2,p}(\Omega)\cap W_{1}^{2,2}(\Omega)$ the solution of the following problem:
\begin{equation*}
\Delta\,\lambda=0\quad\mathrm{in}\,\, \Omega\quad\mathrm{and}\quad \lambda=1\quad\mathrm{on}\,\, \Gamma.
\end{equation*}
It follows from Lemma $3.11$ of $\cite{Vivet_1994}$ that
 $$\Big\langle \dfrac{\partial \lambda}{\partial \textbf{\textit{n}}},1 \Big\rangle_{\Gamma}=\int_{\Gamma}\dfrac{\partial \lambda}{\partial \textbf{\textit{n}}} d\sigma= C_1>0.$$ Now, setting
\begin{equation*}
\widetilde{\textbf{\textit{z}}}=\textbf{\textit{z}}-\dfrac{1}{C_{1}}\left\langle \textbf{\textit{z}}\cdot\textbf{\textit{n}},1 \right\rangle _{\Gamma}\nabla\,\lambda.
\end{equation*}
It is clear that $\widetilde{\textbf{\textit{z}}}\in \textbf{\textit{L}}_{}^{p}(\Omega)$, $\mathrm{div}\,\widetilde{\textbf{\textit{z}}}=0$ in $\Omega$ and that $\left\langle \widetilde{\textbf{\textit{z}}}\cdot\textbf{\textit{n}},1 \right\rangle _{\Gamma}=0$. {\color{blue}Since $p>3$, Theorem~\ref{Theorem Vecteur potentiel} (which requires $1<p<3/2$) cannot be applied here, and we construct the vector potential directly. Since $\left\langle \widetilde{\textbf{\textit{z}}}\cdot\textbf{\textit{n}},1 \right\rangle _{\Gamma}=0$, the Neumann problem $\Delta\,\theta'=0$ in $\Omega'$, $\partial\theta'/\partial\textbf{\textit{n}}=\widetilde{\textbf{\textit{z}}}\cdot\textbf{\textit{n}}$ on $\Gamma$, has a solution $\theta'\in W^{1,p}(\Omega')$ such that $\Vert\nabla\,\theta'\Vert_{L^{p}(\Omega')}\leq C\Vert\widetilde{\textbf{\textit{z}}}\cdot\textbf{\textit{n}}\Vert_{W^{-1/p,p}(\Gamma)}\leq C\Vert\widetilde{\textbf{\textit{z}}}\Vert_{\textbf{\textit{L}}^{p}(\Omega)}$. Let $\widehat{\textbf{\textit{z}}}$ be the field equal to $\widetilde{\textbf{\textit{z}}}$ in $\Omega$ and to $\nabla\,\theta'$ in $\Omega'$. Then $\widehat{\textbf{\textit{z}}}\in \textbf{\textit{L}}^{p}(\R^3)$ and $\mathrm{div}\,\widehat{\textbf{\textit{z}}}=0$ in $\R^3$. Since $p>3$, the Laplace operator $\Delta : W^{1,p}_{0}(\R^3)/\R\longrightarrow W^{-1,p}_{0}(\R^3)$ is an isomorphism (see~\cite{Amrouche_1994}), hence there exists $\textbf{\textit{w}}\in W^{1,p}_{0}(\R^3)$ such that $-\Delta\,\textbf{\textit{w}}=\mathbf{curl}\,\widehat{\textbf{\textit{z}}}$ in $\R^3$ and $\Vert\textbf{\textit{w}}\Vert_{W^{1,p}_{0}(\R^3)}\leq C\Vert\widehat{\textbf{\textit{z}}}\Vert_{\textbf{\textit{L}}^{p}(\R^3)}$. The function $\mathrm{div}\,\textbf{\textit{w}}\in L^{p}(\R^3)$ is harmonic, so $\mathrm{div}\,\textbf{\textit{w}}=0$; then $\mathbf{curl}\,\textbf{\textit{w}}-\widehat{\textbf{\textit{z}}}\in \textbf{\textit{L}}^{p}(\R^3)$ is divergence-free and curl-free, hence harmonic, and thus $\mathbf{curl}\,\textbf{\textit{w}}=\widehat{\textbf{\textit{z}}}$ in $\R^3$. Next, by Theorem~\ref{theorem Neumann}, there exists $\theta\in W^{2,p}_{0}(\Omega)$ such that $\Delta\,\theta=0$ in $\Omega$, $\partial\theta/\partial\textbf{\textit{n}}=\textbf{\textit{w}}\cdot\textbf{\textit{n}}$ on $\Gamma$ and, by \eqref{estimation forte neumann}, $\Vert\theta\Vert_{W^{2,p}_{0}(\Omega)}\leq C\Vert\textbf{\textit{w}}\cdot\textbf{\textit{n}}\Vert_{W^{1-1/p,p}(\Gamma)}\leq C\Vert\textbf{\textit{w}}\Vert_{W^{1,p}_{0}(\R^3)}$. Then the field $\boldsymbol{\psi}=\textbf{\textit{w}}_{\mid\Omega}-\nabla\,\theta$ belongs to $W_{0}^{1,p}(\Omega)$ and satisfies} 
\begin{equation}\label{conséquence vecteur potentiel 2}
 \widetilde{\textbf{\textit{z}}}=\mathbf{curl}\,\boldsymbol{\psi},\quad\mathrm{div}\,\boldsymbol{\psi}=0\quad\mathrm{in}\,\Omega\quad\mathrm{and}\quad\boldsymbol{\psi}\cdot\textbf{\textit{n}}=0\quad\mathrm{on}\,\Gamma.
\end{equation}
In addition, we have the estimate
\begin{equation}\label{relation de norme entre z tilde et z}
|| \boldsymbol{\psi}||_{W_{0}^{1,p}(\Omega)} \leq C ||\widetilde{\textbf{\textit{z}}}||_{ \textbf{\textit{L}}_{}^{p}(\Omega)} \leq C  ||\textbf{\textit{z}}||_{ \textbf{\textit{L}}_{}^{p}(\Omega)}.
\end{equation}
Then, we obtain that $\boldsymbol{\psi}$ belongs to $\textbf{\textit{V}}^{\,p}_{0,T}(\Omega)$. Since $\boldsymbol{\varphi}$ is $W^{1,p'}_{loc}$ in a neighborhood of $\Gamma$, then $\boldsymbol{\varphi}$ has an $W^{1,p'}_{loc}$ extension in $\Omega'$ denoted by $\widetilde{\boldsymbol{\varphi}}$. Applying Green's formula in $\Omega'$, we obtain
\begin{equation*}
 0=\int_{\Omega'}\mathrm{div}(\mathbf{curl}\,\widetilde{\boldsymbol{\varphi}})\,\mathrm{d}\textbf{\textit{x}}=\left\langle \mathbf{curl}\,\widetilde{\boldsymbol{\varphi}}\cdot\textbf{\textit{n}},1\right\rangle _{\Gamma}=\left\langle \mathbf{curl}\,\boldsymbol{\varphi}\cdot \textbf{\textit{n}},1\right\rangle _{\Gamma}.
\end{equation*} 
Using the fact that $\mathbf{curl}\,\boldsymbol{\varphi}$ in $H_{0}^{\,p'}(\mathrm{div},\Omega)$ and $\lambda$ in $W_{1}^{2,p}(\Omega)\hookrightarrow W_{0}^{1,p}(\Omega)$ and applying \eqref{FG1}, we obtain
\begin{equation}\label{calcul integrale de nabla lambda nul}
0=\left\langle \mathbf{curl}\,\boldsymbol{\varphi}\cdot \textbf{\textit{n}},1\right\rangle _{\Gamma}=\left\langle \mathbf{curl}\,\boldsymbol{\varphi}\cdot \textbf{\textit{n}},\lambda\right\rangle _{\Gamma}=\int_{\Omega}\mathbf{curl}\,\boldsymbol{\varphi}\cdot\nabla\,\lambda\,\mathrm{d}\textbf{\textit{x}}.
\end{equation}
Using  \eqref{calcul integ zero 2} and \eqref{calcul integrale de nabla lambda nul}, we deduce that 

\begin{equation}\label{egalite des integrale total}
\int_{\Omega}\mathbf{curl}\,\boldsymbol{\varphi}\cdot\textbf{\textit{g}}\,\mathrm{d}\textbf{\textit{x}}=\int_{\Omega}\mathbf{curl}\,\boldsymbol{\varphi}\cdot\textbf{\textit{z}}\,\mathrm{d}\textbf{\textit{x}}=\int_{\Omega}\mathbf{curl}\,\boldsymbol{\varphi}\cdot\widetilde{\textbf{\textit{z}}}\,\mathrm{d}\textbf{\textit{x}}.
\end{equation}
From  \eqref{inegalite decomposition de helmholt 2}, \eqref{relation de norme entre z tilde et z} and \eqref{egalite des integrale total}, we deduce that 
\begin{equation*}
\dfrac{\big{\vert}\int_{\Omega}\mathbf{curl}\,\boldsymbol{\varphi}\cdot\textbf{\textit{g}}\,\mathrm{d}\textbf{\textit{x}}\big{\vert}}{\Vert \textbf{\textit{g}}\Vert_{\textbf{\textit{L}}_{}^{p}(\Omega)}}\leq C\dfrac{\big{\vert}\int_{\Omega}\mathbf{curl}\,\boldsymbol{\varphi}\cdot\widetilde{\textbf{\textit{z}}}\,\mathrm{d}\textbf{\textit{x}}\big{\vert}}{\Vert \widetilde{\textbf{\textit{z}}}\Vert_{\textbf{\textit{L}}_{}^{p}(\Omega)}} = C\dfrac{\big{\vert}\int_{\Omega}\mathbf{curl}\,\boldsymbol{\varphi}\cdot\mathbf{curl}\,\boldsymbol{\psi}\,\mathrm{d}\textbf{\textit{x}}\big{\vert}}{\Vert \mathbf{curl}\,\boldsymbol{\psi}\Vert_{\textbf{\textit{L}}_{}^{p}(\Omega)}}.
\end{equation*}
\textcolor{blue}{Since $\mathbf{curl}\,\boldsymbol{\psi}=\widetilde{\textbf{\textit{z}}}$, estimate \eqref{relation de norme entre z tilde et z} gives $\Vert \boldsymbol{\psi}\Vert_{W_{0}^{\,1,p}(\Omega)}\leq C\Vert \mathbf{curl}\,\boldsymbol{\psi}\Vert_{\textbf{\textit{L}}^{p}(\Omega)}$, and} we obtain
\begin{equation*}
 \dfrac{\big{\vert}\int_{\Omega}\mathbf{curl}\,\boldsymbol{\varphi}\cdot\textbf{\textit{g}}\,\mathrm{d}\textbf{\textit{x}}\big{\vert}}{\Vert \textbf{\textit{g}}\Vert_{\textcolor{blue}{\textbf{\textit{L}}^{p}}(\Omega)}}
\leq  C\dfrac{\big{\vert}\int_{\Omega}\mathbf{curl}\,\boldsymbol{\varphi}\cdot\mathbf{curl}\,\boldsymbol{\psi}\,\mathrm{d}\textbf{\textit{x}}\big{\vert}}{\Vert \boldsymbol{\psi}\Vert_{W_{0}^{\,1,p}(\Omega)}},
\end{equation*}
 and the Inf-Sup Condition \eqref{condition inf sup 2} follows immediately from \eqref{dualite 2}\textcolor{blue}{, since $\Vert\cdot\Vert_{\textbf{\textit{Z}}^{\,q}_{T}(\Omega)}\leq C\Vert\cdot\Vert_{W^{1,q}_{0}(\Omega)}$ for $q=p,p'$}.
\end{proof}\medskip
%%%%%%%%%%%%%%%%%%%%%%%%%%%%%%%%%%%%%%%%%%%%%%%%

\section{ Existence of weak solution for $p>3$}
We prove in this sequel the existence and the uniqueness of weak solutions for the following Problem:
\begin{eqnarray*}\label{problem Curl}
(\mathcal{S}_T)\begin{cases}
-\Delta\,\textbf{\textit{u}}+\nabla\,{\pi}=\textbf{\textit{f}}\quad and\quad
\mathrm{div}\,\textbf{\textit{u}}={\chi}\quad\text{ in }\,\Omega,\\
\quad\textbf{\textit{u}}\cdot \textbf{\textit{n}}=g\,\,\,\,and\quad
\mathbf{curl}\,\textbf{\textit{u}}\times\textbf{\textit{n}}=\textbf{\textit{h}}\quad\text{ on }\,\Gamma.
\end{cases}
\end{eqnarray*}
The following result concerns the existence and uniqueness of the weak solution of the problem $(\mathcal{S}_T)$ where $\chi=g=0$. {\color{blue}Let us introduce the kernel
\begin{equation*}
\textbf{\textit{K}}^{\,p}_{T}(\Omega)=\Big\lbrace \textbf{\textit{v}}\in W^{1,p}_{0}(\Omega);\,\,\mathrm{div}\,\textbf{\textit{v}}=0,\,\,\mathbf{curl}\,\textbf{\textit{v}}=\textbf{0}\,\,\text{ in }\,\Omega,\,\,\textbf{\textit{v}}\cdot\textbf{\textit{n}}=0\,\,\text{ on }\,\Gamma\Big\rbrace.
\end{equation*}
Since $p>3$, the constants belong to $W^{1,p}_{0}(\Omega)$ and this kernel is not reduced to $\{\textbf{0}\}$. Indeed, for $i=1,2,3$, let $w_i\in W^{1,2}_{0}(\Omega)$ be the solution of $\Delta\,w_i=0$ in $\Omega$, $\partial w_i/\partial\textbf{\textit{n}}=n_i$ on $\Gamma$ (see~\cite{Amrouche_JMPA_1997}). Since $\int_{\Gamma}n_i\,d\sigma=0$, we have $w_i=O(r^{-2})$ at infinity, so that $\nabla\,w_i\in W^{1,p}_{0}(\Omega)$ and the field $\textbf{\textit{e}}_i-\nabla\,w_i$, which tends to $\textbf{\textit{e}}_i$ at infinity, is a nonzero element of $\textbf{\textit{K}}^{\,p}_{T}(\Omega)$. Hence $(\textbf{\textit{e}}_i-\nabla\,w_i,0)$ is a nonzero solution of $(\mathcal{S}_T)$ with $\textbf{\textit{f}}=\textbf{0}$, $\chi=g=0$ and $\textbf{\textit{h}}=\textbf{0}$, and the uniqueness can only hold up to an element of $\textbf{\textit{K}}^{\,p}_{T}(\Omega)$.}
\begin{theo}\label{solution chi=0}
Assume that $p>3$, suppose that $\chi=g=0$. Then, for any $\textbf{\textit{f}}\in[\mathring{H}_{-1}^{\,\,p'}(\mathrm{div},\Omega)]'$ and $\textbf{\textit{h}}\in W^{-1/p,p}(\Gamma)$ \textcolor{blue}{such that $\textbf{\textit{h}}\cdot\textbf{\textit{n}}=0$ on $\Gamma$}, the problem $(\mathcal{S}_T)$ has a solution $(\textbf{\textit{u}},{\pi})\in(W^{1,p}_{0}(\Omega)\times {{L}}^{p}(\Omega))$\textcolor{blue}{, where $\textbf{\textit{u}}$ is unique up to an element of $\textbf{\textit{K}}^{\,p}_{T}(\Omega)$ and $\pi$ is unique}. In addition, we have the following estimate:
\begin{eqnarray}\label{estimation Stokes non homogene}
\Vert \textbf{\textit{u}}\Vert_{W^{1,p}_{0}(\Omega)\textcolor{blue}{/\textbf{\textit{K}}^{\,p}_{T}(\Omega)}}+\Vert {\pi}\Vert_{{{L}}^{p}(\Omega)}\leqslant C(\Vert \textbf{\textit{f}}\Vert_{[\mathring{H}_{-1}^{\,\,p'}(\mathrm{div},\Omega)]'}+\Vert \textbf{\textit{h}}\Vert_{W^{-1/p,p}(\Gamma)}).
\end{eqnarray}
\end{theo}
\begin{proof}
On the first hand, observe that Problem $(\mathcal{S}_T)$ is reduced to the following variational problem:
\begin{eqnarray}\label{problem variationnel}
\begin{cases}
Find\quad \textbf{\textit{u}}\in\textbf{V}_{0,T}^{p}(\Omega)\quad such \quad that,\\
\quad \displaystyle{\int}_{\Omega}\textbf{curl}\textbf{\textit{u}}.\textbf{curl}\boldsymbol{\varphi} d\textbf{x}=\langle\textbf{\textit{f}},\boldsymbol{\varphi}\rangle_{\Omega}+\left\langle\textbf{\textit{h}},\boldsymbol{\varphi} \right\rangle_{\Gamma}, \forall \boldsymbol{\varphi}\in \textbf{\textit{V}}^{p'}_{0,T}(\Omega)
\end{cases}
\end{eqnarray}
Let us introduce the space 
\begin{equation*}
\label{def.Dsigma}
\mathcal{D}_\sigma(\Omega)=\Big\{\boldsymbol{\varphi}\in\mathcal{D}(\Omega),\,\mathrm{div }\,\boldsymbol{\varphi}=0\Big\}.
\end{equation*}
Indeed, every solution of $(\mathcal{S}_T)$ also solves \eqref{problem variationnel}. Conversely, let us take $\boldsymbol{\varphi}\in\mathcal{D}_{\sigma}(\Omega)$ as the test function in \eqref{problem variationnel}, then we get

\begin{eqnarray}
\langle-\Delta\,\textbf{\textit{u}}-\textbf{\textit{f}},\boldsymbol{\varphi}\rangle_{\mathcal{D}_{\sigma}'(\Omega)\times \mathcal{D}_{\sigma}(\Omega)}=0.
\end{eqnarray}
By De Rham theorem, there exist a function $\pi\in L^{p}(\Omega)$ such that $$-\Delta\,\textbf{\textit{u}}+\nabla\,\pi=\textbf{\textit{f}}\quad \text{ in } \Omega.$$

\noindent Moreover, by the fact that $\textbf{\textit{u}}$ belongs to the space $\textbf{\textit{V}}^{p}_{0,T}(\Omega)$, we have $\mathrm{div}\,\textbf{\textit{u}}=0$ in $\Omega$, $\textbf{\textit{u}}\cdot\textbf{\textit{n}}=0$ on $\Gamma$. The remainder boundary condition $\mathbf{curl}\,\textbf{\textit{u}}\times\textbf{\textit{n}}=\textcolor{blue}{\textbf{\textit{h}}}$ on $\Gamma$ is implicitly contained in \eqref{problem variationnel}. Observe that since $\nabla\,\pi$ is elements of $[\mathring{H}^{p'}_{-1}(\mathrm{div},\Omega)]'$, it is the same for $\Delta\,\textbf{\textit{u}}$. Since, $\mathcal{D}(\Omega)$ is dense in $\mathring{H}^{p'}_{-1}(\mathrm{div},\Omega)$, it is clear then that for any $\boldsymbol{\varphi}\in \mathring{H}^{p'}_{-1}(\mathrm{div},\Omega)$ 
\begin{eqnarray*}
\langle\nabla\,\pi,\boldsymbol{\varphi}\rangle_{\Omega}=-\displaystyle\int_{\Omega}\pi\mathrm{div}\,\boldsymbol{\varphi}\,d\textbf{\textit{x}},
\quad\forall\boldsymbol{\varphi}\in \mathring{H}^{p'}_{-1}(\mathrm{div};\Omega).
\end{eqnarray*}
\noindent Particularly $$\langle\nabla\,\pi,\boldsymbol{\varphi}\rangle_{\Omega}=0\quad\text{if}\quad\boldsymbol{\varphi}\in {\textbf{V}}^{p'}_{0,\textcolor{blue}{T}}(\Omega).$$
Moreover, if $\boldsymbol{\varphi}\in \textbf{\textit{V}}^{p'}_{0,T}(\Omega)$, using the Green formula \eqref{formule de grenn dans E^p} we have

\begin{eqnarray}\label{formul vari H div}
\langle -\Delta\,\textbf{\textit{u}},\boldsymbol{\varphi}\rangle_{[\mathring{H}_{-1}^{p'}(\mathrm{div},\Omega)]'\times\mathring{H}_{-1}^{p'}(\mathrm{div},\Omega)}=\int_{\Omega}\mathbf{curl}\,\textbf{\textit{u}}\,.\,\mathbf{curl}\,\varphi\,d\textbf{x}\textcolor{red}{-}\langle\mathbf{curl}\,\textbf{\textit{u}}\times\textbf{\textit{n}},\varphi\rangle_{\Gamma}
\end{eqnarray}
Therefore, from \eqref{problem variationnel} and \eqref{formul vari H div} we deduce that for all $\boldsymbol{\varphi}\in \textbf{\textit{V}}^{p'}_{0,T}(\Omega)$
\begin{eqnarray*}
\langle\mathbf{curl}\,\textbf{\textit{u}}\times\textbf{\textit{n}},\boldsymbol{\varphi}\rangle_{\Gamma}=\langle \textcolor{blue}{\textbf{\textit{h}}},\boldsymbol{\varphi} \rangle_{\Gamma}
\end{eqnarray*}
Let now $\boldsymbol{\mu}$ any element of the space $W^{1-1/p',p'}(\Gamma)$. So, there exists an element $\boldsymbol{\varphi}\in W^{1,p'}_{0}(\Omega)$ such that $\mathrm{div}\,\boldsymbol{\varphi}=0$ in $\Omega$ and $\boldsymbol{\varphi}=\boldsymbol{\mu}_{\tau}$ on $\Gamma$. Then, \textcolor{blue}{since $\textbf{\textit{h}}\cdot\textbf{\textit{n}}=0$ on $\Gamma$}, we have
\begin{eqnarray*}
\langle\mathbf{curl}\,\textbf{\textit{u}}\times\textbf{\textit{n}},\boldsymbol{\mu}\rangle_{\Gamma}-\langle \textcolor{blue}{\textbf{\textit{h}}},\boldsymbol{\mu} \rangle_{\Gamma} &=& \langle\mathbf{curl}\,\textbf{\textit{u}}\times\textbf{\textit{n}},\boldsymbol{\mu}_{\tau}\rangle_{\Gamma}-\langle \textcolor{blue}{\textbf{\textit{h}}},\boldsymbol{\mu}_{\tau} \rangle_{\Gamma}\\
&=& \langle\mathbf{curl}\,\textbf{\textit{u}}\times\textbf{\textit{n}},\boldsymbol{\varphi}\rangle_{\Gamma}-\langle \textcolor{blue}{\textbf{\textit{h}}},\boldsymbol{\varphi} \rangle_{\Gamma}
\end{eqnarray*}
This implies that $\mathbf{curl}\,\textbf{\textit{u}}\times\textbf{\textit{n}}=\textcolor{blue}{\textbf{\textit{h}}}$ on $\Gamma$. \\

\noindent This implies that problem $(\mathcal{S}_T)$ and problem \eqref{problem variationnel} are equivalent. Now, to solve Problem \eqref{problem variationnel}, we use the Inf-Sup condition \eqref{condition inf sup 2}. We consider the bilinear form $\textbf{a} :$ $\textbf{\textit{V}}_{0,T}^{\,p}(\Omega)\times \textbf{\textit{V}}_{0,T}^{\,p'}(\Omega)\longrightarrow \R$ such that
\begin{eqnarray*}
\textbf{a}(\textbf{\textit{u}},\varphi)=\displaystyle{\int}_{\Omega}\mathbf{curl}\,{\textbf{\textit{u}}}\cdot\mathbf{curl}\,\boldsymbol{\varphi}\,\,\mathrm{d}\textbf{\textit{x}}.
\end{eqnarray*}
Let consider the following mapping $\boldsymbol{\ell}:\,\textbf{\textit{V}}^{\,p'}_{0,T}(\Omega)\longrightarrow\,\R$ such that $\boldsymbol{\ell}(\boldsymbol{\varphi})=\left\langle \textbf{\textit{f}},\boldsymbol{\varphi}\right\rangle _{\Omega}+\langle\textbf{\textit{h}},\boldsymbol{\varphi}\rangle_{\Gamma}$. It is clear that $\boldsymbol{\ell}$ belongs to $(\textbf{\textit{V}}^{\,p'}_{0,T}(\Omega))'$ and according to \textcolor{blue}{Lemma \ref{lemme condition inf sup 2}, Theorem \ref{Babuchca brezi} and} Remark \ref{remark sur c inf sup} \textcolor{blue}{$iii)$}, there exists a solution $\textbf{\textit{u}}\in \textbf{\textit{V}}^{\,p}_{0,T}(\Omega)$ of Problem \eqref{problem variationnel}\textcolor{blue}{, unique up to an element of $\mathrm{Ker}\,A$, where $A\in\mathcal{L}(\textbf{\textit{V}}^{\,p}_{0,T}(\Omega);(\textbf{\textit{V}}^{\,p'}_{0,T}(\Omega))')$ is the operator associated with $\textbf{a}$}. Due to \textcolor{blue}{\cite[Theorem 5.1]{Louati_Meslameni_Razafison} (recall that $\textbf{\textit{V}}^{\,p}_{0,T}(\Omega)\subset\textbf{\textit{Z}}^{\,p}_{T}(\Omega)$)}, we prove that this solution $\textbf{\textit{u}}$ belongs \textcolor{blue}{to} $W^{1,p}_{0}(\Omega)$. {\color{blue}Let us show that $\mathrm{Ker}\,A=\textbf{\textit{K}}^{\,p}_{T}(\Omega)$. Clearly $\textbf{\textit{K}}^{\,p}_{T}(\Omega)\subset\mathrm{Ker}\,A$. Conversely, let $\textbf{\textit{u}}\in\mathrm{Ker}\,A$ and $\textbf{\textit{g}}\in\boldsymbol{\mathcal{D}}(\Omega)$. Let $\chi_0\in\mathring{W}^{1,2}_{0}(\Omega)$ be the solution of $\Delta\,\chi_0=\mathrm{div}\,\textbf{\textit{g}}$ in $\Omega$, $\chi_0=0$ on $\Gamma$ (see~\cite{Amrouche_JMPA_1997}), and let $\lambda$ be the function introduced in the proof of Lemma~\ref{lemme condition inf sup 2}. Outside a ball containing $\overline{\Omega'}$ and the support of $\textbf{\textit{g}}$, the functions $\chi_0$ and $\lambda$ are harmonic and tend to zero at infinity, so that $\chi_0=a/r+O(r^{-2})$ and $\lambda=b/r+O(r^{-2})$, with $b\neq0$ because $\int_{\Gamma}\partial\lambda/\partial\textbf{\textit{n}}\,d\sigma=C_1\neq0$. The function $\chi=\chi_0-(a/b)\lambda$ is constant on $\Gamma$, satisfies $\Delta\,\chi=\mathrm{div}\,\textbf{\textit{g}}$ in $\Omega$, $\chi=O(r^{-2})$ and $\nabla\,\chi=O(r^{-3})$ at infinity, hence $\chi\in W^{1,p'}_{0}(\Omega)$. The field $\textbf{\textit{z}}=\textbf{\textit{g}}-\nabla\,\chi$ belongs to $\textbf{\textit{L}}^{p'}(\Omega)$ and is divergence-free; since $p>3$, the constants belong to $W^{1,p}_{0}(\Omega)$ and \eqref{FG1} gives $\left\langle \textbf{\textit{z}}\cdot\textbf{\textit{n}},1\right\rangle_{\Gamma}=0$. By Theorem~\ref{Theorem Vecteur potentiel} applied with the exponent $p'\in\,]1,3/2[$, there exists $\boldsymbol{\varphi}\in W^{1,p'}_{0}(\Omega)$ such that $\textbf{\textit{z}}=\mathbf{curl}\,\boldsymbol{\varphi}$, $\mathrm{div}\,\boldsymbol{\varphi}=0$ in $\Omega$ and $\boldsymbol{\varphi}\cdot\textbf{\textit{n}}=0$ on $\Gamma$. Thus $\boldsymbol{\varphi}\in\textbf{\textit{V}}^{\,p'}_{0,T}(\Omega)$ and $\int_{\Omega}\mathbf{curl}\,\textbf{\textit{u}}\cdot\textbf{\textit{z}}\,d\textbf{x}=\textbf{a}(\textbf{\textit{u}},\boldsymbol{\varphi})=0$. Moreover, as in the proof of Lemma~\ref{lemme condition inf sup 2}, $\left\langle \mathbf{curl}\,\textbf{\textit{u}}\cdot\textbf{\textit{n}},1\right\rangle_{\Gamma}=0$, and since $\chi$ is constant on $\Gamma$, \eqref{FG1} yields $\int_{\Omega}\mathbf{curl}\,\textbf{\textit{u}}\cdot\nabla\,\chi\,d\textbf{x}=\left\langle \mathbf{curl}\,\textbf{\textit{u}}\cdot\textbf{\textit{n}},\chi\right\rangle_{\Gamma}=0$. Therefore $\int_{\Omega}\mathbf{curl}\,\textbf{\textit{u}}\cdot\textbf{\textit{g}}\,d\textbf{x}=0$ for all $\textbf{\textit{g}}\in\boldsymbol{\mathcal{D}}(\Omega)$, i.e. $\mathbf{curl}\,\textbf{\textit{u}}=\textbf{0}$ and $\textbf{\textit{u}}\in\textbf{\textit{K}}^{\,p}_{T}(\Omega)$. Consequently, if $(\textbf{\textit{u}}_1,\pi_1)$ and $(\textbf{\textit{u}}_2,\pi_2)$ are two solutions of $(\mathcal{S}_T)$, then $\textbf{\textit{u}}_1-\textbf{\textit{u}}_2\in\textbf{\textit{K}}^{\,p}_{T}(\Omega)$, hence $\Delta(\textbf{\textit{u}}_1-\textbf{\textit{u}}_2)=\textbf{0}$, $\nabla(\pi_1-\pi_2)=\textbf{0}$ and $\pi_1=\pi_2$ because $\pi_1-\pi_2\in L^{p}(\Omega)$.} It follows from Remark \ref{remark sur c inf sup} \textcolor{blue}{$iii)$} that

\begin{eqnarray}\label{estimation inf-sup}
\Vert \textbf{\textit{u}}\Vert_{W^{1,p}_{0}(\Omega)\textcolor{blue}{/\textbf{\textit{K}}^{\,p}_{T}(\Omega)}}\leqslant C(\Vert \textbf{\textit{f}}\Vert_{[\mathring{H}_{-1}^{\,\,p'}(\mathrm{div},\Omega)]'}+\Vert \textbf{\textit{h}}\Vert_{W^{-1/p,p}(\Gamma)}).
\end{eqnarray}
Since $\nabla\,\pi=\textbf{\textit{f}}+\Delta\,\textbf{\textit{u}}$ \textcolor{blue}{and $\Delta\,\textbf{\textit{v}}=\textbf{0}$ for all $\textbf{\textit{v}}\in\textbf{\textit{K}}^{\,p}_{T}(\Omega)$}, then we have the following estimate
\begin{eqnarray}\label{estimation pi}
{\Vert\pi\Vert_{L^{p}(\Omega)}\leqslant  C\Vert\nabla\,\pi\Vert_{W^{-1,p}_{0}(\Omega)}}\leqslant C\big(\Vert \textbf{\textit{f}}\Vert_{[\mathring{H}_{-1}^{p'}(\mathrm{div},\Omega)]'}+\Vert\textbf{\textit{u}}\Vert_{W^{1,p}_{0}(\Omega)\textcolor{blue}{/\textbf{\textit{K}}^{\,p}_{T}(\Omega)}}\big)
\end{eqnarray}
The estimate~\eqref{estimation Stokes non homogene} follows from  \eqref{estimation inf-sup} and \eqref{estimation pi}.
\end{proof}
\noindent We can also solve the Stokes problem when the divergence operator does not vanish and $g\neq 0$.
\begin{coro}\label{solution faible chi diferent de zero}
Assume that $p>3$. Let $\textbf{\textit{f}}$, $\chi$, $g$, $\textbf{\textit{h}}$ such that
\begin{equation*}
\textbf{\textit{f}}\in[ \mathring{H}_{-1}^{\,p'}(\mathrm{div},\,\Omega)]',\,\chi \in L^{\,p}(\Omega),\,\, g \in W^{\,1-1/p,p}(\Gamma)\,\,\,\mathrm{ and}\,\, \,\,\textbf{\textit{h}}\in W^{\,-1/p,p}(\Gamma)\textcolor{blue}{\,\,\mathrm{with}\,\,\textbf{\textit{h}}\cdot\textbf{\textit{n}}=0\,\,\mathrm{on}\,\,\Gamma}.
\end{equation*}

\noindent Then, the Stokes problem $(\mathcal{S}_T)$ has a solution $(\textbf{\textit{u}}\,,\,\pi)\in W_{0}^{\,1,p}(\Omega)\times L^{\,p}(\Omega)$\textcolor{blue}{, where $\textbf{\textit{u}}$ is unique up to an element of $\textbf{\textit{K}}^{\,p}_{T}(\Omega)$ and $\pi$ is unique,} and we have:
\begin{equation}\label{estimation Stokes non homogene chi}
 \Vert\textbf{\textit{u}}\,\Vert_{W_{0}^{\,1,p}(\Omega)\textcolor{blue}{/\textbf{\textit{K}}^{\,p}_{T}(\Omega)}}+ \Vert\,\pi\,\Vert_{L^{\,p}(\Omega)} \leqslant  C\left(  \Vert\,\textbf{\textit{f}}\,\Vert_{[\mathring{H}_{-1}^{\,p'}(\mathrm{div},\,\Omega)]'}+\Vert \,\chi\,\Vert_{L^{\,p}(\Omega)}+\Vert g\Vert_{W^{\,1-1/p,p}(\Gamma)}+ \Vert\textbf{\textit{h}}\Vert_{W^{\,-1/p,p}(\Gamma)}\right) .
\end{equation}
\end{coro}
\begin{proof}
Let $\chi\in L^{p}(\Omega)$ and $g\in W^{1-1/p,p}(\Gamma)$. We solve the following Neumann problem in $\Omega$:
\begin{equation}\label{Neumann avec done chi m=1}
\textcolor{blue}{\Delta\,\theta}=\chi\quad\mathrm{in}\,\, \Omega,\quad \frac{\partial \theta}{\partial \textbf{\textit{n}}}=g\quad\mathrm{on}\,\, \Gamma.
\end{equation}
It follows from Theorem \ref{theorem Neumann}, that the Problem \eqref{Neumann avec done chi m=1} has a solution $\theta$ in $W_{0}^{\,2,p}(\Omega)$ and we have:
\begin{equation}\label{estimation de vivet}
||\theta||_{W_{0}^{\,2,p}(\Omega)} \leqslant C \left( ||\chi||_{L^{\,p}(\Omega)}+||g||_{W^{1-1/p,p}(\Gamma)}\right) .
\end{equation}
Setting $\textbf{\textit{z}}=\textbf{\textit{u}}-\nabla\,\theta$, then Problem $(\mathcal{S}_T)$  becomes: Find $(\textbf{\textit{z}},\pi) \in W_{0}^{1,p}(\Omega)\times L^{\,p}(\Omega)$ such that 
\begin{equation}
\begin{cases}
-\Delta\,\textbf{\textit{z}}+\nabla\,\pi=\textbf{\textit{f}}+\nabla\,\chi\quad\mathrm{and}\quad
\mathrm{div}\,\textbf{\textit{z}}=0 &\quad\mathrm{in}\,\Omega,\\
\textbf{\textit{z}}\cdot\textbf{\textit{n}}=0\quad\mathrm{and}\quad\mathbf{curl}\,\textbf{\textit{z}}\times\textbf{\textit{n}}=\textbf{\textit{h}}&\quad\mathrm{on}\,\mathrm{\Gamma},\\ 
\end{cases}
\end{equation}
Observe that $\textbf{\textit{f}}+\nabla\,\chi$ belongs to $[ \mathring{H}_{-1}^{\,p'}(\mathrm{div},\,\Omega)]'$. According to Theorem \ref{solution chi=0}, this problem has a solution $(\textbf{\textit{z}}\,,\,\pi)\in W_{0}^{\,1,p}(\Omega)\times L^{\,p}(\Omega)$\textcolor{blue}{, $\textbf{\textit{z}}$ being unique up to an element of $\textbf{\textit{K}}^{\,p}_{T}(\Omega)$}. Thus $\textbf{\textit{u}}=\textbf{\textit{z}}+\nabla\,\theta$ belongs to $W_{0}^{\,1,p}(\Omega)$ and estimate \eqref{estimation Stokes non homogene chi} follows from \eqref{estimation Stokes non homogene} and \eqref{estimation de vivet}. \textcolor{blue}{Finally, the difference of two solutions of $(\mathcal{S}_T)$ is a solution of $(\mathcal{S}_T)$ with $\textbf{\textit{f}}=\textbf{0}$, $\chi=g=0$ and $\textbf{\textit{h}}=\textbf{0}$, so that the uniqueness follows from Theorem \ref{solution chi=0}.}
\end{proof}

%%%%%%%%%%%%%%%%%%%%%%%%%%%%%%%%%%%%%

%%%%%%%%%%%%%%%%%%%%%%%%%%%%%

\end{document}